\documentclass[reqno]{amsart}
\usepackage{amssymb}
\usepackage{amscd,amsfonts}
\usepackage{tikz}
\usepackage{pgfplots}
\usetikzlibrary{automata,positioning,calc,trees}
\usetikzlibrary{intersections,pgfplots.fillbetween}
\usetikzlibrary{shapes.geometric, arrows}
\usepackage{graphicx,tikz}
\usepackage{pgf,tikz,pgfplots}
\usepackage{mathrsfs}
\usepackage{enumerate}
\usepackage[shortlabels]{enumitem}
\usepackage{mathrsfs} 
\usepackage{amssymb,amsmath,amsthm,color}
\usepackage{caption,subcaption}
\usepackage{hyperref}
\usepackage{thmtools}
\usepackage{cleveref}
\usepackage[nobysame,initials,alphabetic,bibtex-style]{amsrefs} 
\usepackage{url}
\usepackage{setspace} 
\usepackage{float}
\usepackage{makecell}

\BibSpec{article}{%
	+{}  {\PrintAuthors}                {author}
	+{,} { \textit}                     {title}
	+{.} { }                            {part}
	+{:} { \textit}                     {subtitle}
	+{,} { \PrintContributions}         {contribution}
	+{.} { \PrintPartials}              {partial}
	+{,} { }                            {journal}
	+{}  { \textbf}                     {volume}
	+{}  { \PrintDatePV}                {date}
	+{,} { \issuetext}                  {number}
	+{,} { \eprintpages}                {pages}
	+{,} { }                            {status}
	+{,} { \url}                        {url}    
	+{,} { \PrintDOI}                   {doi}
	+{,} { available at \eprint}        {eprint}
	+{}  { \parenthesize}               {language}
	+{}  { \PrintTranslation}           {translation}
	+{;} { \PrintReprint}               {reprint}
	+{.} { }                            {note}
	+{.} {}                             {transition}
}
\BibSpec{book}{%
	+{}  {\PrintAuthors}                {author}
	+{,} { \textit}                     {title}
	+{.} { }                            {part}
	+{:} { \textit}                     {subtitle}
	+{,} { \PrintContributions}         {contribution}
	+{.} { \PrintPartials}              {partial}
	+{,} { }                            {series}
	+{}  { \textbf}                     {volume}
	+{}  { \PrintDatePV}                {date}
	+{,} { \issuetext}                  {number}
	+{,} { \eprintpages}                {pages}
	+{,} { }                            {status}
	+{,} { \url}                        {url}    
	+{,} { available at \eprint}        {eprint}
	+{}  { \parenthesize}               {language}
	+{}  { \PrintTranslation}           {translation}
	+{;} { \PrintReprint}               {reprint}
	+{.} { }                            {note}
	+{.} {}                             {transition}
	+{,} { \PrintDOI}                   {doi}
}

\newcommand{\R}{{\mathbb {R}}}

\newcommand{\N}{{\mathbb N}}
\newcommand{\Z}{{\mathbb Z}}
\newcommand{\C}{{\mathbb C}}

\newcommand{\Rr}{{\mathfrak R}}
\newcommand{\A}{{\mathcal{A}}}

\newcommand{\Ss}{{\mathbf{S}}}
\newcommand{\rr}{{\mathfrak{r}}}

\newtheorem{theorem}{Theorem}[section]
\newtheorem{lemma}[theorem]{Lemma}

\newtheorem{proposition}[theorem]{Proposition}

\newtheorem*{theorem*}{Theorem}

\theoremstyle{definition}
\newtheorem{definition}[theorem]{Definition}
\newtheorem{example}[theorem]{Example}
\newtheorem{remark}[theorem]{Remark}

\numberwithin{equation}{section}

\title[Variational estimates for discrete spherical averages] {Variational estimates for discrete spherical averages: $\ell^p$-boundedness and related arithmetic obstructions}

 \keywords{}
 
\author[Bhojak]{Ankit Bhojak}
\address{Ankit Bhojak}
\email{ankitb@iiserb.ac.in}

\author[Samanta]{Siddhartha Samanta}
\address{Siddhartha Samanta}
\email{siddhartha21@iiserb.ac.in}

\author[Shrivastava]{Saurabh Shrivastava}

\address{Saurabh Shrivastava}
\email{saurabhk@iiserb.ac.in}
 \address{Department of Mathematics\\
 	Indian Institute of Science Education and Research Bhopal\\
 	Bhopal-462066, India.}

\subjclass[2020]{Primary 11P55, 11D72, Secondary 42B25, 37A46 }
\keywords{}

\begin{document}

\begin{abstract}
		In this article, we investigate variational inequalities for the family of discrete spherical averages. We establish $\ell^p$-boundedness of the $r$-variation operators for $1<p<\infty$ and $r>2$ along sufficiently sparse sequences of factorial radii. Moreover, we also study arithmetic obstructions to these variational estimates and construct counterexamples to show that such boundedness properties for the corresponding variational operators cannot hold for lacunary sequences of radii.
\end{abstract}  

\maketitle

 \section{Introduction}\label{sec:intro} 
    \subsection{Discrete spherical averages}
	For a function $f:\mathbb{Z}^{n} \rightarrow \mathbb{C}$ in $\ell^{p}(\mathbb{Z}^{n})$, $1\leq p\leq \infty$ and $\lambda \geq 1 $, the discrete spherical average is defined by
	\[A_{\lambda}f(x) = \frac{1}{r_n(\lambda)}\sum_{\substack{y\in\mathbb{Z}^{n}:\;|y| = \lambda}}f(x-y),\]
	where $r_n(\lambda) = |\{ y \in \mathbb{Z}^{n} : |y| = \lambda \}|$ satisfying $r_{n}(\lambda)\approx \lambda^{n-2}$. In the seminal work \cite{MSW2002}, Magyar-Stein-Wainger employed the Hardy-Littlewood-Ramanujan circle method and a sampling principle to establish the $\ell^p(\Z^n)$-bounds for the full maximal operator $A^*$, defined as 
	\[A^*f(x)=\sup_{\lambda\geq 1}|A_{\lambda}f(x)|,\]
	in the sharp range of indices $p>\frac{n}{n-2}$ and $n\geq 5$. Ionescu \cite{Ionescu} obtained a restricted weak-type estimate for $A^*$ at the endpoint $p=\frac{n}{n-2}$. A further refinement of these results was obtained by Lacey-Kesler-Mena \cite{Lacey1}, where they proved a sparse domination result for the maximal operator $A^*$.
	
	We say a sequence of numbers $\mathbf{L}=\{\lambda_l\}_{l=1}^{\infty}$ is lacunary if there exists a $c>1$ such that $\lambda_{l+1}>c\lambda_l$ for all $l\in\N$. Zienkiewicz constructed a probabilistic counterexample to show that the discrete spherical maximal function $A^*_{\mathbf{L}}f(x)=\sup_{\lambda_l\in\mathbf{L}}|A_{\lambda_l}f(x)|$ is unbounded on $\ell^{p}(\mathbb{Z}^{n})$ for $1\leq p<\frac{n}{n-1}$ for a certain lacunary sequence $\mathbf{L}$. We refer to \cite{Cook&Hughes1} and \cite{Krause} for details. On the positive side, the operator $A^*_{\mathbf{L}}$ was shown to be bounded on $\ell^{p}(\mathbb{Z}^{n})$ for all lacunary sequences $\mathbf{L}$ and $p>\frac{n-2}{n-3}$ by Kesler-Lacey-Mena \cite{LaceyLacunarySpherical} and Cook-Hughes \cite{Cook&Hughes1}. However, it remains open whether the maximal operator $A^*_{\mathbf{L}}$ is bounded on $\ell^p(\Z^n)$ in the range $\frac{n}{n-1}\leq p\leq\frac{n-2}{n-3}$. It is of interest to know for which sequences $\mathbf{S}$ the discrete maximal function $A_\mathbf{S}^*$ is bounded in $\ell^p(\Z^n)$ for all $p>1$. 
    \begin{definition}\label{Definition sup-lacunary seqn}
		We call a sequence $\mathbf{S}=\{\lambda_{k}\}_{k=1}^{\infty}$ sparse, if there exists $\mu_k$ such that $\lambda_{k}=\mu_{k}!,\;k\in\N$, with
		$\lim\limits_{k\rightarrow\infty}\frac{\log \mu_k}{\log k}=\infty.$
	\end{definition} 
    It was shown in \cites{LaceyLacunarySpherical, Cook2} that the maximal functions $A_\mathbf{S}^*$ are indeed bounded in $\ell^p(\Z^n)$ for all $p>1$ when $\mathbf{S}$ is a sparse sequence. We also refer to \cite{bcss} for sparse bounds and \cite{endPointMaximal} for endpoint weak-type bounds for certain Orlicz spaces near $L^1$ corresponding to discrete maximal functions over sparse sequences. In this article, we aim to study variational inequalities associated with discrete spherical averages.
	\subsection{Variational norms for a family of operators}
    One of the foundational result on variational inequalities concerns the study of $r$-variations associated with a general bounded martingale sequence, initiated in the work of L\'{e}pingle \cite{variationformartingales} for $r>2$. We note that such an inequality may fail for $r\leq2$, see \cites{Qian1998,JW2004}. Bourgain's contribution \cite{Pointwiseergodictheorems} was a notable advancement in this direction, where he reduced the study of variational inequalities to estimating jump operators associated with the corresponding family of functions. Since then, variation and jump inequalities for broad classes of convolution operators have been a major focus of research in harmonic analysis and ergodic theory. For a systematic survey of the development of this subject, we refer to \cites{variationalinequalitycontinuous,Variationalongprimesandpolynomials,variationalestimatesradon,DimensionfreeestimatesdiscreteHL,MSW2022,BootstrapVariationStein,MSZ2020,Bootstrap,VariationBoundSpherical,endPoint} and references therein.
	
	Let $\mathbf{D}$ be a subset of $\N$, and $f:\Z^n\rightarrow\C$ be a function in $\ell^p(\Z^n)$ with $1\leq p\leq\infty$. We recall that for any family of operators $\mathcal{A}_\mathbf{D}:=\{A_{\lambda}:\lambda\in D\}$ acting on $f$ and $r\in [1,\infty]$, the $r$-variation seminorm corresponding to the family $\mathcal{A}_\mathbf{D}f$ is defined by
	\begin{align}
		V_{r}(\mathcal{A}_\mathbf{D}f)(x) := 
        \begin{cases}
            \sup\limits_{L\in\N}\sup\limits_{\substack{\lambda_l\in\mathbf{D},\;l=1,\cdots,L-1\\\lambda_1<\dots<\lambda_L}}\left(\sum\limits_{l=1}^{L-1}|A_{\lambda_{l+1}}f(x)-A_{\lambda_{l}}f(x)|^r\right)^{\frac{1}{r}}, & 0<r<\infty, \\
	           \sup\limits_{\substack{\lambda_0,\lambda_1\in\mathbf{D}:\\\lambda_0<\lambda_1}}\left|A_{\lambda_{1}}f(x)-A_{\lambda_{0}}f(x)\right|, & r=\infty,
        \end{cases} 
	\label{definition of variation} 
	\end{align}
    where the former supremum is taken over all finite increasing sequences $\{\lambda_1,\dots,\lambda_L\}$ belonging to $D$, and by convention it is equal to zero if $D$ is empty. For $r\geq 1$, the $r$-variation norm is defined by   
    \begin{align}
       \mathcal{V}_{r}(\mathcal{A}_\mathbf{D}f)(x):=\sup\limits_{\lambda_l\in D}|A_{\lambda_l}f(x)|+V_{r}(\mathcal{A}_\mathbf{D}f)(x).\label{variation norm}
    \end{align}
    
	
    \section{Main results}\label{sec:mainresults} The purpose of this article is to establish the $\ell^{p}(\mathbb{Z}^{n})$-bounds of $ r$-variation operators associated with discrete spherical averages. More precisely, we are interested in proving an inequality of the form 
	\begin{equation}\label{eqn:variation}
		\|V_r(\mathcal{A}_\mathbf{D} f)\|_{\ell^p(\Z^n)}\lesssim \|f\|_{\ell^p(\Z^n)},
	\end{equation}
	for a sequence $\mathbf{D}\subset\N$. 
	
	We first note that by an elementary parity argument, the inequality \eqref{eqn:variation} cannot hold for the sequence $\mathbf{P}=\{4^j+\textbf{1}_{\{\text{j is odd}\}}\}_{j\geq 1}$. The parity argument essentially relies on the existence of two subsequences of $\mathbf{P}$ with odd and even entries respectively and the details are provided in \Cref{example1}. As a consequence, the $r$-variation inequality \eqref{eqn:variation} cannot hold for an arbitrary lacunary sequence of radii. In fact, we can exploit this parity obstruction further to show that the variational inequality fails even for the dyadic sequence $\mathbf{L}=\{2^j\}_{j\in\N}$. We have the following.
	\begin{proposition}\label{prop:variationunbounded}
		For $n\geq5$ and $p,r\geq1$, there does not exist $C>0$, such that 
		\[\|V_r(\mathcal{A}_\mathbf{\mathbf{L}} f)\|_{\ell^p(\Z^n)}\lesssim \|f\|_{\ell^p(\Z^n)},\]
		holds uniformly for all $f\in\ell^p(\Z^n)$.
	\end{proposition}
	The proof is based on a parity obstruction and an equidistribution property (\Cref{lem: equidistribution result}) of lattice points on the sphere over congruence classes modulo 5. The details of the proof are given in \Cref{proof:variationunbounded}.
	\begin{remark}[\textbf{A contrast with the continuum setting}] We would like to emphasize the contrast between the discrete and continuum setting concerning their respective variational inequalities. Let $n\geq 2$ and $\sigma$ be the surface measure on the unit sphere $\mathbb{S}^{n-1}\subset\R^n$. For $f:\R^n\to\C$, consider the family of spherical averages $\{\mathbf{A}_tf:\;t>0\}$ defined by
	\begin{align}
		\mathbf{A}_{t}f(x)=\int_{\mathbb{S}^{n-1}}f(x-t\theta)d \sigma (\theta), \quad t>0.\label{continuous spherical avg} 
	\end{align} 
    It is known that the corresponding full maximal function $\mathbf{A}^*f(x)=\sup_{t>0}|\mathbf{A}_{t}f(x)|$ is bounded in $L^p(\R^n)$ for $p>\frac{n}{n-1}$, see \cite{Stein1,Bourgain2}. The corresponding lacunary maximal function $\mathbf{A}_{2^\Z}^*f(x)=\sup_{j\in\Z}|\mathbf{A}_{2^j}f(x)|$ is bounded \cite{Calderon,DuoanRubio} in $L^p(\R^n)$ for $1<p\leq\infty$. Moreover, Jones-Seeger-Wright \cite{variationalinequalitycontinuous} showed that for $r>2$, the $r-$variations associated with the continuous spherical averages over lacunary sets of radii are bounded in $L^p(\R^n)$ for the same range of $p$ as the corresponding lacunary spherical maximal functions, except at the endpoint $p=\infty$. In contrast, the $r-$variation operators associated with discrete spherical averages behaves differently as their maximal counterparts, in view of \Cref{prop:variationunbounded}. We also refer to \cite{BootstrapVariationStein,VariationBoundSpherical} for further results concerning variational inequalities for spherical averages in the continuous setting. 
	\end{remark} 
	
	The above counterexamples suggest that it would be a futile effort to prove variational bounds of the form \eqref{eqn:variation} for sequences $\mathbb{D}$ consisting of atleast two distinct subsequences lying in different congruence classes modulo a fixed number, in fact the inequality \eqref{eqn:variation} would not hold true in such cases. Based on this observation, it would be natural to investigate if \eqref{eqn:variation} is true when the sequence $\mathbb{D}$ is of highly composite nature, for instance, when $\mathbb{D}$ is sparse in the sense of \Cref{Definition sup-lacunary seqn}. The main result of this paper concerns the $\ell^p-$bounds for the $r-$variations associated to the discrete spherical averages along the sparse sequences. More precisely, we prove the following result.
	
	\begin{theorem}\label{thm:varSphereSparse}
		Let $n\geq 5$, $r\in(2,\infty)$ and $\Ss$ be a sparse sequence in the sense of \Cref{Definition sup-lacunary seqn}. Then, for any $p\in (1,\infty)$, we have
		\begin{equation}\label{mainineq:Vr}
			\|V_r(\mathcal{A}_\Ss f)\|_{\ell^p(\Z^n)}\lesssim \|f\|_{\ell^p(\Z^n)},
		\end{equation}
		for all functions $f\in\ell^p(\Z^n)$.
	\end{theorem}
	Let $\mathcal{A}_\Ss=\{A_{\lambda_l}:\;\lambda_l\in\Ss\}$ be the family of discrete spherical averages over a sparse sequence $\Ss$.
    We note that
	\[A_{\Ss}^*f(x):=\sup\limits_{\lambda\in\Ss}|A_{\lambda} f(x)|\leq |A_{\lambda_0}^{\Rr}f(x)|+2V_r(\mathcal{A}_\Ss f)(x),\]    
  	holds for any $\lambda_0\in\N$ and $1\leq r<\infty$. Thus, the $\ell^p-$bounds for the $r-$variations of the family $\A_{\Ss}$ imply the corresponding bounds for the sparse maximal operator $A_{\Ss}^*$. 

	The proof of \Cref{thm:varSphereSparse} is based on estimating the jump functions associated with the discrete spherical averages introduced in \Cref{sec:jump}. The jump functions are estimated by employing an appropriate high-low decomposition of the corresponding multiplier based on the Hardy-Littlewood circle method as in \cite{LaceyLacunarySpherical} and a transference principle of \cite{MSW2002}. We would like to emphasize that the choice of sparse sequence of radii in the theorem above is not merely due to techniques of the proof. It is motivated by the arithmetic obstruction discussed in the counterexamples above. \Cref{sec:proof of main theorem} is devoted to the proof of \Cref{thm:varSphereSparse}.
    \subsection{Application to ergodic averages and counterexample}\label{sec:counterexample}
	In this section, we discuss an application of our main theorem in proving variational bounds for the ergodic spherical averages. For $n\geq5$, let $(X, \mathcal{B}(X),\mu)$ be a $\sigma$-finite measure space and $T=(T_1,T_2,\dots,T_n)$ is a family of commuting, measure-preserving, and invertible transformations. We refer to $(X, \mathcal{B}(X),\mu,T)$ as a dynamical system. For every $f\in L^1(X),$ we define the ergodic spherical averaging operator as follows.
	\[\mathbb{A}_{\lambda}f(x)=\frac{1}{r_{n}(\lambda)}\sum_{y\in\Z^n:|y|=\lambda}f\left(T_1^{y_1}\circ T_2^{y_2}\circ\cdots\circ T_n^{y_n}x\right),\]
	where $\lambda$ is a positive integer. These averages were considered by Magyar in \cite{Magyar2}, where an $L^2$-ergodic theorem was established for such averages under the strong ergodicity condition on the underlying transformations $T_1,\dots,T_n$.
	For a sparse sequence $\Ss$, denote $\mathfrak{A}_\Ss=\{\mathbb{A}_{\lambda}:\lambda\in \Ss\}$. We establish  the following variational estimate for the family $\mathfrak{A}_\Ss f=\{\mathbb{A}_{\lambda}f:\lambda\in \Ss\}$.
	\begin{theorem}\label{thm:main ergodic Vr}
        Let $n\geq5$, $r\in (2,\infty)$ and $\Ss$ be a sparse sequence in the sense of \Cref{Definition sup-lacunary seqn}. Suppose $(X, \mathcal{B}(X),\mu,T)$ be a dynamical system. Then, for any $p\in (1,\infty)$, we have
		\begin{equation*}
			\|V_{r}(\mathfrak{A}_\Ss f)\|_{L^p(X)}\lesssim \|f\|_{L^p(X)}.
		\end{equation*}
		for all functions $f\in L^p(X)$.
	\end{theorem}
	We remark that the above theorem is a direct consequence of \Cref{thm:varSphereSparse} and the following transference principle.
	\begin{theorem}\label{thm:variational transference principle}
		Let $D\subset\N$, $1\leq p<\infty$ and $r>2$. Suppose the following estimate
		\begin{align}
			\|V_{r}(\mathcal{A}_\lambda f:\;\lambda\in D)\|_{\ell^p(\Z^n)}
			&\lesssim \|f\|_{\ell^p(\Z^n)}\label{variation for BM}
		\end{align}
		holds for all $f\in\ell^p(\Z^n)$. Then for every $g\in L^p(X)$ we have 
		\begin{align}\label{variation for ergodic BM}
			\|V_{r}(\mathfrak{A}_\lambda f:\;\lambda\in D)\|_{L^p(X)}\lesssim \|f\|_{L^p(X)}.
		\end{align}
	\end{theorem}
	The proof of \Cref{thm:variational transference principle} is given in \Cref{sec:ergodic transference}.

	Moreover, we will also show that the $r$-variation seminorm of the family of ergodic spherical averages taken along the dyadic sequence $\{2^l\}_{l\geq 1}$ fails to be bounded on $\ell^p(\Z^n)$ for $n\geq5$ and $p,r\geq1$. We have the following.
    \begin{proposition}\label{thm:variation unbounded}
        For the lacunary sequence $\mathbf{L}=\{2^l\}_{l\geq 1}$, denote the family $\mathfrak{A}_\mathbf{L}=\{\mathbb{A}_{\lambda}:\lambda\in \mathbf{L}\}$. For any $r\in [1,\infty)$ and $p\in [1,\infty)$, there exists a dynamical system $(X, \mathcal{B}(X),\mu,T)$ such that $V_r(\mathfrak{A}_\mathbf{L})$ is unbounded on $L^p(X)$.
    \end{proposition}
	The proof of \Cref{thm:variation unbounded} is presented in \Cref{proof:variationunbounded}.
    \section{Notations and Auxiliary results}\label{sec:aux}
	\subsection{Notations}
	We use the following notation throughout the paper.
	\begin{enumerate}[i)]
		\item The notation $\mathscr{F}_{\Z^n}$ and $\mathscr{F}_{\R^n}$ are used to denote the Fourier transform on $\Z^{n}$ and $\R^n$ respectively. 
		\item The notation $Y\lesssim Z$ (or $Z\gtrsim Y$) means that there exists an absolute constant $0<C<\infty$ independent of the parameters on which $Y$ and $Z$ depend, such that $Y\leq CZ$ (or $Y\geq CZ$). Also, $Y\approx Z$ means that $Y\lesssim Z$ and $Y\gtrsim Z$.
		\item We denote $e(t)=e^{2\pi it}$, for $t\in\R$.
		\item The group $\mathbb{Z}_{q}$ denotes $\mathbb{Z}/q\mathbb{Z}$, and $U_{q}$ denotes the multiplicative group $\mathbb{Z}_{q}^{\times}$ with the understanding that $U_{1}=\mathbb{Z}_{1}=\{0\}$. We also denote $\Z_{q}^{n}=\Z_{q}\times\Z_{q}\times\dots\times\Z_{q}~(\text{n-times})$.
        \item Sometimes we use the notation $V_r(\{A_\lambda f,\;\lambda\in\Ss\})(x)$ instead of $V_r(\mathcal{A}_\Ss f)(x)$ to emphasize the set that indexes the scale.
        \item For a function $f$ defined on $\R^n$, and for any $a\in\R\setminus\{0\}$, the dilation is defined by $f_{a}(x)=\frac{1}{a^n}f\left(\frac{x}{a}\right)$.
	\end{enumerate}

    \subsection{Jump inequalities}\label{sec:jump} The proof of \Cref{thm:varSphereSparse} factors through the jump inequalities corresponding to the family $\mathcal{A}_\Ss$. We define the $t$-jump corresponding to the family $\mathcal{A}_\mathscr{S}$, denoted by $\mathcal{J}_{t}(\mathcal{A}_\Ss)$, as the supremum of all integers $M\in\N$ for which there is an increasing sequence $0<u_1<v_1\leq u_2<v_2\leq\cdots\leq u_M<v_M$ with $u_i,v_i\in\Ss$ such that
    \[|A_{u_l}f(x)-A_{v_l}f(x)|>t\] 
   for each $l=1,\dots,M$. Next, we consider a modified notion of $t$-jump. Specifically, we define $J_t(\mathcal{A}_{\Ss}f)(x)$ as the supremum of all nonnegative integer $M$ for which there exists a sequence $0<s_1<\dots <s_M$ satisfying the property that $|A_{s_{l+1}}f(x)-A_{s_{l}}f(x)|>t,\;\text{for } l=1,\dots, M-1$. It is straightforward to verify that 
   \begin{align}
       J_t(\mathcal{A}_{\Ss}f)(x)\leq\mathcal{J}_t(\mathcal{A}_{\Ss}f)(x)\leq 2J_{t/2}(\mathcal{A}_{\Ss}f)(x),\label{comparison of jump}
   \end{align}
   thereby showing that the two definitions are interchangeable up to some constant. The advantage of $\mathcal{J}_t$ lies in its subadditive property; i.e., for any two family of operators $\A^1$ and $\A^2$ we have
   \begin{align*}
       \mathcal{J}_{t}(\A^1f+\A^2f)(x)\leq \mathcal{J}_{t/2}(\A^1f)(x)+\mathcal{J}_{t/2}(\A^2f)(x).
   \end{align*}
   The operator $J_t$ admits a pointwise comparison with the $q$-variation, namely $t(J_t(\A))^{1/r}\leq V_r(\A f)(x)$. Combining this with \eqref{comparison of jump} one can deduce that 
   \begin{align*}
       t(\mathcal{J}_t(\A f)(x))^{1/r}\lesssim V_r(\A f)(x),
   \end{align*}
   where the constant is independent of $t$. 
   
   The following lemma states that the $\ell^p-$jump inequality implies the $\ell^p-$boundedness of the $r-$variation operator for $r>2$. This lemma is proven by Jones-Seeger-Wright for the general measure space $(X,\mu)$. We therefore adopt it directly from \cite[Lemma 2.1]{variationalinequalitycontinuous} and apply it to both the spaces $\ell^p(\Z^n)$ and $L^p(\R^n)$.
	\begin{lemma}\cite[Lemma 2.1, pp.6719]{variational inequality continuous}\label{thm:jump implies variation b}
    	Let $\mathcal{T}$ be a family of linear operators that maps $L^p(X,\mu)$ to the set of measurable functions on $X$. Suppose
    	\begin{align}
        	\sup_{t>0}\|t(\mathcal{J}_{t}(\mathcal{T}f))^{1/s}\|_{L^p}&\lesssim\|f\|_{L^p}\label{t-jump}
    	\end{align}
    	holds for $p_0<p<p_1$ and $p_0<s<p_1$ with all $f$ in a dense subset of $L^p(X,\mu)$. Then for $r>s$ and for all $f\in L^p(X,\mu)$, we have
    	\begin{align*}
        	\|V_{r}(\mathcal{T}f)\|_{L^p}&\lesssim\|f\|_{L^p}.
    	\end{align*}
    \end{lemma}
	\subsection{Transference principle}\label{sec:trans}
	Let $K$ be a suitable distribution kernel and consider the convolution operator $Tf(x)=f* K(x)$. Denote $\widehat{K}(\xi)=m(\xi)$, thus $(Tf)^{\wedge}(\xi)=m(\xi)\widehat{f}(\xi)$. By restricting $K$ to $\Z^n$, i.e., $K_{dis}=K\vert_{\mathbb{Z}^{n}},$ consider the discrete operator $T_{dis}(f)=f* K_{dis}$. Assume that $m(\xi)$ is supported in $\left[-\frac{1}{2q},\frac{1}{2q}\right]^{n},$ where $q$ is a fixed integer $q \geq 1$. Consider the periodic function
	\[\tilde{m}^{q}(\xi)=\sum_{b\in\mathbb{Z}^{n}}m\left(\xi-\frac{b}{q}\right).\] 
    Define $T^{q}_{dis}$ as the convolution operator on $\Z^n$, having $\tilde{m}^{q}$ as its Fourier multiplier. We remark that, the two operators $T^{q}_{dis}$ and $(T^{q})_{dis}$ are not same. In particular, the convolution kernel of $(T^{q})_{dis},$ which is the Fourier inversion of the multiplier $m(\xi/q)$ is given by $K^q(x)=q^nK(qx),\;x\in\Z^n$. Whereas, the convolution kernel of $T^{q}_{dis}$ is given by 
	\begin{align*}
		\mathcal{K}^q(x)=
		\begin{cases}
			q^nK(x) & \text{if } x\in q\Z^n\\[6pt]
			0       & \text{if } x\notin q\Z^n.
		\end{cases}
	\end{align*}
	For a detailed discussion, we refer the interested reader to \cite{MSW2002}. Let, $B_1$ and $B_2$ be a pair of finite dimensional Banach spaces and $m$ taking its values in $\mathcal{L}(B_1,B_2)$-the space of bounded linear operators from $B_1$ to $B_2$. We require the following transference result from \cite{MSW2002}. 
	\begin{proposition}\label{prop:transference banach valued}\cite[Corollary 2.1, pp.196]{MSW2002}
		Let $T$ and $T^{q}_{dis}$ be the operators defined as above. Assume that $T$ is a bounded operator from $L^{p}_{B_1}(\mathbb{R}^{n})\rightarrow L^{p}_{B_2}(\mathbb{R}^{n})$.  Then\[\|T_{dis}^{q}\|_{\ell^p_{B_1}\rightarrow\ell^p_{B_2}}\lesssim \|T\|_{L^p_{B_1}\rightarrow L^p_{B_2}}.\] 
   \end{proposition}
   Note the following theorem concerning the $L^p$ bounds of $t$-jump of certain family of averages.
    \begin{theorem}\cite[Theorem 1.1, pp.6714]{variationalinequalitycontinuous}
        Let $\mu$ be a compactly supported finite Borel measure on $\R^n$ and $\mathcal{U}=\{\mathcal{U}_l\}$, where $\mathcal{U}_l f=f*\mu_{2^l}=\int_{\R^n}f(x-2^ly)d\mu(y)$. Then if $\mu$ satisfies that $\mathcal{F}_{\R^n}(\mu)(\xi)=O(|\xi|^{-\alpha})$ for some $\alpha>0$, we have 
		\begin{align*}
        	\|t(\mathcal{J}_{t}(\mathcal{U}f))^{1/2}\|_{L^p}&\lesssim\|f\|_{L^p},
    	\end{align*}
		for any $1<p<\infty$, uniformly in $t>0$ and $f\in L^p(\R^n)$.
    \end{theorem}
	Next, observe that $d\sigma$ satisfies the necessary conditions of the above theorem, hence we have the following lemma regarding the $t$-jump of the family $\mathscr{C}_{\mathbf{L}}f$.
    \begin{lemma}
        Let, $\mathbf{L}:=\{2^l\}_{l\geq 1}$ and $\mathscr{C}_{\mathbf{L}}f:=\{\mathbb{A}_{t}f=f*d\sigma_{t}:t\in\mathbf{L}\}$. Then for any $1<p<\infty$ and $r>2$, we have
		\begin{align*}
        	\|t(\mathcal{J}_{t}(\mathscr{C}_{\mathbf{L}}f))^{1/2}\|_{L^p}&\lesssim\|f\|_{L^p},
    	\end{align*}
		uniformly in $t>0$ and $f\in L^p(\R^n)$.
	\end{lemma}
	Now, combining this result along with \Cref{thm:jump implies variation b} and the fact that the lacunary spherical maximal function is bounded on $L^p(\R^n)$ for $1<p<\infty$, it follows that, for all $r>2$,
	\begin{align}
		\left\|\mathcal{V}_r(\mathbf{A}_{t}f:t\in\mathbf{L})\right\|_{L^p(\R^n)}
		\leq C_{p,r}\|f\|_{L^p(\R^n)}.\label{continuous variation}
	\end{align}


	\section{Variational inequalities over sparse sequences: Proof of \Cref{thm:varSphereSparse}}\label{sec:proof of main theorem}
	First, note that $\mu_l\approx l^{\alpha(l)},$ where $\alpha(l)\rightarrow \infty$ as $l\rightarrow \infty$. Therefore, for $\epsilon>0$, there exists a constant $N_{\circ}>0$ such that $\mu_{N^{\epsilon}}>N^{2}$ for $N>N_{\circ}$. 
	For any $\epsilon\in(0,1)$, a finitely supported function $f=1_F$, 
    and any integer $N$ we decompose the fixed scale operator $A_{\lambda_l}$ as
	\begin{align*}
		A_{\lambda_l}f(x)
		&= A_{\lambda_l}^1f(x)+A_{\lambda_l}^2f(x)
	\end{align*} 
	such that the $t$-jump of the family of operators $\A^1:=\{A_{\lambda_l}^1:l\geq1\}$ and $\A^2:=\{A_{\lambda_l}^2:l\geq 1\}$ satisfy the following inequalities:
	\begin{eqnarray}
		\left\|t\left(\mathcal{J}_{t}(\mathcal{A}^1f)\right)^{1/\rr}\right\|_{\ell^{p}} \lesssim N^{\epsilon}|F|^{\frac{1}{p}}, \label{l^p estimate}\\
		\left\|t\left(\mathcal{J}_{t}(\mathcal{A}^2f)\right)^{1/\rr}\right\|_{\ell^{2}} \lesssim N^{-(\frac{n}{2}-2)}|F|^{1/2},\label{l^2 estimate}
	\end{eqnarray}
	for any $\rr>2$. Note that the constants in the inequalities above depend on $p\text{ and }n$ but are independent of $t$. Using the property of $t$-jump, the estimates \eqref{l^p estimate} and \eqref{l^2 estimate} yield 
	\begin{eqnarray} 
		\left|\left\{x\in\Z^n:|t(\mathcal{J}_{t}(\mathcal{A}f)(x))^{1/\rr}|>\alpha\right\}\right|
		\lesssim\left(\frac{N^{p\epsilon}}{\alpha^{p}}+\frac{N^{-(n-4)}}{\alpha^2}\right)|F|\label{restricted weak type estm}.
	\end{eqnarray}
	Next, to optimize the right-hand side of \eqref{restricted weak type estm}, we choose $N=\alpha^{\frac{p-2}{p\epsilon+n-4}}$. Thereafter, substituting it into \eqref{restricted weak type estm} and letting $\epsilon\rightarrow 0$, we get the desired estimate 
	\begin{align}
		\left\|t\left(\mathcal{J}_{t}(\mathcal{A}f)\right)^{1/\rr}\right\|_{\ell^{p,\infty}}
		\lesssim \|f\|_{\ell^p},\label{restricted weak type estm final}
	\end{align}
	for the characteristic function $f=1_F$ of a finite set $F\subset\Z^n,\;\rr>2$, and $1<p<\infty$. Then, by the standard interpolation, we obtain that the $t$-jump of the family $\mathcal{A}$ satisfies \eqref{t-jump}, for any $1<p<\infty$ and $\rr=2+\delta$ with $\delta>0$. Finally, invoking \Cref{thm:jump implies variation b}, we deduce \Cref{thm:varSphereSparse}. Therefore, we need to produce the decomposition of the family $\mathcal{A}f$ described as above. We carry out the details in the following section.
	\subsection{Decomposition of the discrete spherical averages:} We begin with the decomposition of the operator as follows: 
	\begin{align}
		A_{\lambda_l}f(x)
		&= \boldsymbol{A}_{1,\lambda_l}f(x)+\boldsymbol{A}_{2,\lambda_l}f(x),\label{first decomposition}
	\end{align}
    where $$\boldsymbol{A}_{1,\lambda_l}f(x):={\bf 1}_{ \left\{l:\lambda_l\leq\lambda_{\lceil N^{\epsilon}\rceil}\right\}}(l)A_{\lambda_l}f(x)\quad\text{and}\quad \boldsymbol{A}_{2,\lambda_l}f(x):={\bf 1}_{ \left\{l:\lambda_l>\lambda_{\lceil N^{\epsilon}\rceil}\right\}}(l)A_{\lambda_l}f(x).$$ 
    Consider $$\mathscr{A}_{1}f=\left\{A_{\lambda_l}f:l\leq \lceil N^\epsilon\rceil\right\}\quad \text{and}\quad \mathscr{A}_{2}f=\left\{A_{\lambda_l}f:l> \lceil N^\epsilon\rceil\right\}.$$ 

	\subsubsection{Decomposition of the family $\mathscr{A}_{2}f$:} For a fixed $\lambda_l\geq \lambda_{\lceil N^{\epsilon}\rceil}$, let $a_{\lambda_l}$ be the multiplier associated with the operator $A_{2,\lambda_l}$. We then employ the Hardy-Littlewood-Ramanujan circle method to decompose the multiplier $a_{\lambda_l}$ into an infinite sum, each term corresponding to a rational $a/q$, for $q\geq 1$ with $a\in U_q$ as follows:
	\begin{align*}
		a_{\lambda_l}(\xi)=m_{\lambda_l}(\xi)+\left(a_{\lambda_l}(\xi)-m_{\lambda_l}(\xi)\right),\;\text{where}
	\end{align*}
	\begin{align}
		m_{\lambda_l}(\xi)
		&=c_n\sum_{q=1}^{\infty}\sum_{a\in U_{q}}e_{q}(-\lambda_{l}^2 a)\sum_{b\in \mathbb{Z}_{q}^{n}}G_{q}(a,b)\mathcal{F}_{\R^n}(\zeta_{q})\left(\xi-\frac{b}{q}\right)\mathcal{F}_{\R^n}(d\sigma_{\lambda_l})\left(\xi-\frac{b}{q}\right).\label{main multiplier}
	\end{align}
	Here in the expression above, $c_n=\pi^{n/2}/\Gamma(n/2)$ and $\zeta$ is a radial Schwartz function on $\mathbb{R}^{n}$ such that $\mathcal{F}_{\R^n}(\zeta)(\xi)$ is identically $1$ for $|\xi|\leq\frac{1}{2n}$, $0$ if $|\xi|\geq \frac{1}{n}$, and $0\leq\mathcal{F}_{\R^n}(\zeta)\leq 1$ with $\mathcal{F}_{\R^n}(\zeta_{q})(\xi)=\mathcal{F}_{\R^n}(\zeta)(q\xi)$. Furthermore, $\mathcal{F}_{\R^n}(d\sigma_{\lambda_l})$ is the Fourier transform of the normalised surface measure $d\sigma_{\lambda_l}$, supported on the sphere of radius $\lambda_l$ in $\mathbb{R}^{n}$, while $G_{q}(a,b)$ denotes the normalized Gauss sum, defined as
	\[G_{q}(a,b)=\frac{1}{q^{n}}\sum_{m\in\mathbb{Z}_{q}^{n}}e\left(|m|^2\frac{a}{q}+\frac{m\cdot b}{q} \right),\]  where $a\in U_{q}$ and $b\in\mathbb{Z}_{q}^{n}$. 
	We have the following estimate~\cite[p. 200]{MSW2002} for the quantity $G_{q}(a,b)$: 
	\[|G_{q}(a,b)|\lesssim q^{-n/2},\quad\text{for}\quad (a,q)=1.\]
	We denote the operators associated with the multipliers $\left(a_{\lambda_l}(\xi)-m_{\lambda_l}(\xi)\right)$ and $m_{\lambda_l}(\xi)$ respectively as $E_{\lambda_l}$ and $M_{\lambda_l}$. Moving forward, we decompose the fixed scale operator $M_{\lambda_l}$ as:
	\begin{align*}
		M_{\lambda_l}f(x)
		&= c_n M_{1,\lambda_l}f(x)+c_n M_{2,\lambda_l}f(x),
	\end{align*}
	where the operators  $M_{1,\lambda_l}$ and $M_{2,\lambda_l}$ are associated with the multipliers $m_{\lambda_l}^{1}$ and $m_{\lambda_l}^{2}$ respectively, defined as
	\begin{align*}
		m_{1,\lambda_l}(\xi)
		&=\sum_{q=1}^{\mu_l}\sum_{a\in U_{q}}\sum_{b\in \mathbb{Z}_{q}^{n}}G_{q}(a,b)\mathcal{F}_{\R^n}(\zeta_{q})\left(\xi-\frac{b}{q}\right)\mathcal{F}_{\R^n}(d\sigma_{\lambda_l})\left(\xi-\frac{b}{q}\right),\;\text{and}\\
		m_{2,\lambda_l}(\xi)
		&=\sum_{q>\mu_l}\sum_{a\in U_{q}}e_{q}(-\lambda_{l}^2 a)\sum_{b\in \mathbb{Z}_{q}^{n}}G_{q}(a,b)\mathcal{F}_{\R^n}(\zeta_{q})\left(\xi-\frac{b}{q}\right)\mathcal{F}_{\R^n}(d\sigma_{\lambda_l})\left(\xi-\frac{b}{q}\right).
	\end{align*}
	Note that, in view of the availability of sufficient decay in the $\ell^2$ estimate of the operator $M_{2,\lambda_l}$, we can easily control the $\ell^2$ norm of the square function associated to the family $\mathcal{M}:=\{M_{2,\lambda_l}:\lambda_l\geq\lambda_{\lceil N^{\epsilon}\rceil}\}$, which satisfies the required decay as stated in \eqref{l^2 estimate}. Therefore, we now turn to the low‑frequency term $m_{1,\lambda_l}(\xi)$. Observe that, because of the super lacunary nature of the sequence $\lambda_l$, the exponential term will be identically 1 in this case. Our aim is to employ the transference principle to estimate the variation norm associated with the family of operators arising from these low‑frequency terms. However, because of the scale dependence present in the range of $q$, the multipliers of this form are not suitable to apply transference principle. We thus proceed with a further decomposition to obtain a more favourable structure, as follows:
	\begin{align*}
		m_{1,\lambda_l}(\xi)
		&=\sum_{q=1}^{\infty}\sum_{a\in U_{q}}\sum_{b\in \Z_{q}^{n}}G_{q}(a,b)\mathcal{F}_{\R^n}(\zeta_{q})\left(\xi-\frac{b}{q}\right)\mathcal{F}_{\R^n}(d\sigma_{\lambda_l})\left(\xi-\frac{b}{q}\right)\\
		&-\sum_{q>\mu_l}\sum_{a\in U_{q}}\sum_{b\in \Z_{q}^{n}}G_{q}(a,b)\mathcal{F}_{\R^n}(\zeta_{q})\left(\xi-\frac{b}{q}\right)\mathcal{F}_{\R^n}(d\sigma_{\lambda_l})\left(\xi-\frac{b}{q}\right)\\
		&:=\mathfrak{m}_{1,\lambda_l}(\xi)+\mathfrak{m}_{2,\lambda_l}(\xi).
	\end{align*}
	At this stage, notice that the multiplier $\mathfrak{m}_{2,\lambda_l}(\xi)$ coincides with the form $m_{2,\lambda_l}(\xi)$. Consequently, we can prove that the family of operators associated to this multiplier satisfies \eqref{l^2 estimate}. Next, we further decompose the multiplier $\mathfrak{m}_{1,\lambda_l}(\xi)$ as 
	\begin{align*}
		\mathfrak{m}_{1,\lambda_l}(\xi)
		&:=s_{1,\lambda_l}(\xi)+s_{2,\lambda_l}(\xi),
	\end{align*}
	where
	\begin{align*}
		s_{1,\lambda_l}(\xi)
		&=\sum_{\substack{1\leq q\leq N!\\:q|N!}}\sum_{a\in U_{q}}\sum_{b\in \Z_{q}^{n}}G_{q}(a,b)\mathcal{F}_{\R^n}(\zeta_{q})\left(\xi-\frac{b}{q}\right)\mathcal{F}_{\R^n}(d\sigma_{\lambda_l})\left(\xi-\frac{b}{q}\right),\; \text{ and }\\
		s_{2,\lambda_l}(\xi)
		&=\sum_{\substack{q\nmid N!}}\sum_{a\in U_{q}}\sum_{b\in \Z_{q}^{n}}G_{q}(a,b)\mathcal{F}_{\R^n}(\zeta_{q})\left(\xi-\frac{b}{q}\right)\mathcal{F}_{\R^n}(d\sigma_{\lambda_l})\left(\xi-\frac{b}{q}\right).
	\end{align*}
	Now, the multiplier $s_{2,\lambda_l}$ is favourable for $\ell^2$ estimates, and we can show that every family of operators associated with this multiplier satisfies \eqref{l^2 estimate}. Before proceeding, let us quickly recall the following identity, which will be useful in the subsequent analysis. For any function $g:\left\{\frac{a}{N}: 0\leq a<N\right\}\to\R$ and any integer $N\geq 1$ we have the identity
	\[\sum_{a\in\mathbb{Z}_{N}}g\left(\frac{a}{N}\right)=\sum_{\substack{1\leq q\leq N\\:q|N}}\sum_{a\in U_{q}}g\left(\frac{a}{q}\right).\]  Incorporating this, the remaining multiplier $s_{1,\lambda_l}(\xi)$ can be expressed as 
	\begin{align*}
		s_{1,\lambda_l}(\xi)
		&:=\sum_{\substack{1\leq q\leq N!\\:q|N!}}\sum_{a\in U_{q}}\sum_{b\in \Z_{q}^{n}}G_{q}(a,b)\mathcal{F}_{\R^n}(\zeta_{q})\left(\xi-\frac{b}{q}\right)\mathcal{F}_{\R^n}(d\sigma_{\lambda_l})\left(\xi-\frac{b}{q}\right)\\
		&=\sum_{a\in \Z_{N!}}\sum_{b\in \Z_{N!}^{n}}G_{N!}(a,b)\mathcal{F}_{\R^n}(\zeta_{\frac{N!}{\nu}})\left(\xi-\frac{b}{N!}\right)\mathcal{F}_{\R^n}(d\sigma_{\lambda_l})\left(\xi-\frac{b}{N!}\right).
	\end{align*}
	Here, in the above expression, $\nu$ is the largest common divisor of $N!,\; a$, and $b$. Let us denote the final form of the multiplier $s_{1,\lambda_l}(\xi)$ as $\mathfrak{b}_{\lambda_l}(\xi)$ and the associated convolution operator as $\mathscr{B}_{\lambda_l}$. Now, we further decompose the multiplier $\mathfrak{b}_{\lambda_l}(\xi)$ as
	\begin{align*}
		\mathfrak{b}_{\lambda_l}(\xi)
		&:=\mathfrak{b}_{1,\lambda_l}(\xi)+\mathfrak{b}_{2,\lambda_l}(\xi),
	\end{align*}
	where 
	\begin{align*}
		\mathfrak{b}_{1,\lambda_l}(\xi)
		&=\sum_{a\in \Z_{N!}}\sum_{b\in \Z_{N!}^{n}}G_{N!}(a,b)\mathcal{F}_{\R^n}(\zeta_{2N!})\left(\xi-\frac{b}{N!}\right)\mathcal{F}_{\R^n}(d\sigma_{\lambda_l})\left(\xi-\frac{b}{N!}\right),\;\text{and}\\
		\mathfrak{b}_{2,\lambda_l}(\xi)
		&=\sum_{a\in \Z_{N!}}\sum_{b\in \Z_{N!}^{n}}G_{N!}(a,b)\left\{\mathcal{F}_{\R^n}(\zeta_{\frac{N!}{\nu}})\left(\xi-\frac{b}{N!}\right)-\mathcal{F}_{\R^n}(\zeta_{2N!})\left(\xi-\frac{b}{N!}\right)\right\}\mathcal{F}_{\R^n}(d\sigma_{\lambda_l})\left(\xi-\frac{b}{N!}\right).
	\end{align*}
	The convolution operators related to the multipliers $\mathfrak{b}_{1,\lambda_l}$ and $\mathfrak{b}_{2,\lambda_l}$ are denoted by $\mathscr{B}_{1,\lambda_l}$ and $\mathscr{B}_{2,\lambda_l}$ respectively. The advantage of working with the operator $\mathscr{B}_{1,\lambda_l}$ is that the $ r$-variation norm of the family of operators associated to this is suitable for $\ell^p$ estimate when $p$ is close to $1$, while the family of operators associated to the multiplier $\mathfrak{b}_{2,\lambda_l}(\xi)$ is favourable for $\ell^2$ estimate of the corresponding $r$-variation.
	Hence, the ultimate decomposition of the operator $A_{2,\lambda_l}$ is given by
	\begin{align*}
		A_{2,\lambda_l}f(x)
		&= E_{\lambda_l}f(x)+c_n M_{2,\lambda_l}f(x)+c_n\mathfrak{M}_{2,\lambda_l}f(x)+c_n S_{2,\lambda_l}f(x)+c_n \mathscr{B}_{1,\lambda_l}f(x)+c_n \mathscr{B}_{2,\lambda_l}f(x),
	\end{align*}
	where $\mathfrak{M}_{2,\lambda_l}$ and $S_{2,\lambda_l}$ are the convolution operators associated to the multipliers $\mathfrak{m}_{2,\lambda_l}$ and $s_{2,\lambda_l}$ respectively. To obtain the $\ell^p$ estimate for the $r$-variation associated to the family of operators $\A:=\left\{A_{\lambda_l}:l\geq 0\right\}$, it is now enough to show that the family of operators $\mathscr{A}_{1}$ and $\mathfrak{B}_1:=\left\{\mathscr{B}_{1,\lambda_l}:l\geq\lceil N^\epsilon\rceil\right\}$ satisfy \eqref{l^p estimate} and the family of operators $\mathcal{E}:=\{E_{\lambda_l}:l\geq\lceil N^\epsilon\rceil\},\;\mathcal{M},\; \mathscr{M}:=\{\mathfrak{M}_{2,\lambda_l}:l\geq\lceil N^\epsilon\rceil\},\;\mathcal{S}:=\{S_{2,\lambda_l}:l\geq\lceil N^\epsilon\rceil\},\;\text{and } \mathfrak{B}_2:=\left\{\mathscr{B}_{2,\lambda_l}:l\geq\lceil N^\epsilon\rceil\right\} $ satisfy \eqref{l^2 estimate}. The rest of the proof is devoted to establishing these estimates.
    
	\subsection{The $\ell^{p}$- estimates:}
	\subsubsection{Estimate for the family of averages associated to small scale:} To estimate the $\ell^p$ norm of the $t$-jump associated to the family $\mathscr{A}_1f$, for any $\rr>2$, we dominate $t\left(\mathcal{J}_{t}(\mathscr{A}_1f)\right)^{1/\rr}$ by $t\left(\mathcal{J}_{t}(\mathscr{A}_1f)\right)^{1/2}$ which is further controled by the associated square function. Thereafter, applying Minkowski's inequality twice, for any $1\leq p\leq\infty$, we have
	\begin{align*}
		\|t\left(\mathcal{J}_t(\mathscr{A}_1f)\right)^{1/\rr}\|_{\ell^p}
		&\leq\left\|\left(\sum_{l=1}^{\lceil N^{\epsilon}\rceil}|A_{\lambda_l}f|^2\right)^{1/2}\right\|_{\ell^p}\\
		&\leq\left[\sum_{x\in\Z^n}\left(\sum_{l=1}^{\lceil N^{\epsilon}\rceil}\left|\frac{1}{r_{n}(\lambda_l)}\sum_{\substack{y:|y|=\lambda_l}}f(x-y)\right|\right)^{p}\right]^{1/p}\\
		&\leq \frac{1}{r_{n}(\lambda_l)}\sum_{l=1}^{\lceil N^{\epsilon}\rceil}\left[\sum_{x\in\Z^n}\left|\sum_{\substack{y:|y|=\lambda_l}}f(x-y)\right|^{p}\right]^{1/p}\\
		&\leq (N^{\epsilon}+1)\|f\|_{\ell^p}.
	\end{align*}

	\subsubsection{Estimate for the family $\mathfrak{B}_{1}f$:} We have, the multiplier of the convolution operator $\mathscr{B}_{1,\lambda_l}$ is 
	\begin{align*}
		\mathfrak{b}_{1,\lambda_l}(\xi)
		&:=\sum_{a\in\Z_{N!}}\sum_{b\in \Z_{N!}^{n}}G_{N!}(a,b)\mathcal{F}_{\R^n}(\zeta_{2N!})\left(\xi-\frac{b}{N!}\right)\mathcal{F}_{\R^n}(d\sigma_{\lambda_l})\left(\xi-\frac{b}{N!}\right).
	\end{align*}
	Now, to estimate the $\rr$-variation of the family of operators $\mathfrak{B}_{1}$, we first write 
	\begin{align*}
		\sum_{a\in\Z_{N!}}\sum_{b\in \Z_{N!}^{n}}G_{N!}(a,b)\mathcal{F}_{\R^n}(\zeta_{2N!})\left(\xi-\frac{b}{N!}\right)\mathcal{F}_{\R^n}(d\sigma_{\lambda_l})\left(\xi-\frac{b}{N!}\right)
		=t_{\lambda_l}(\xi)u(\xi),
	\end{align*}
	where
    \begin{align*}
        t_{\lambda_l}(\xi)
        &=\sum_{b\in \mathbb{Z}_{N!}^{n}}\mathcal{F}_{\R^n}(\zeta_{2N!})\left(\xi-\frac{b}{N!}\right)\mathcal{F}_{\R^n}\left(d\sigma_{\lambda_l}\right)\left(\xi-\frac{b}{N!}\right),\;\text{and}\\
        u(\xi)
        &=\sum_{a\in\mathbb{Z}_{N!}}\sum_{b\in \mathbb{Z}_{N!}^{n}}G_{N!}(a,b)\mathcal{F}_{\R^n}(\zeta_{N!})\left(\xi-\frac{b}{N!}\right)\mathcal{F}_{\Z^n}(f)(\xi).
    \end{align*}
     Therefore, we have
	\begin{align*}	
		V_{\rr}\left(\mathscr{B}_{1,\lambda_l}f:\lambda_l\geq\lambda_{\lceil N^\epsilon\rceil}\right)(x)
		=V_{\rr}\left(T_{\lambda_l}(\mathcal{U}f):\lambda_l\geq\lambda_{\lceil N^\epsilon\rceil}\right)(x),
	\end{align*}
	where $T_{\lambda_l}$ and $\mathcal{U}$ are the convolution operators associated with the multipliers $t_{\lambda_l}(\xi)$ and $u(\xi)$ respectively. Now, observe that the multiplier $t_{\lambda_l}$ corresponds to a family of convolution operators from $\ell^p(\Z^n)$ (scalar-valued) to $\ell^{p}_{B_L}(\Z^n)$ where 
    $$B_L=\left\{b=(b_{\lambda_l})_{\lambda_l\in\Lambda_L}:\ b_{\lambda_l}\in\mathbb{C},\;\|b\|_{\mathcal{V}_{\rr}}<\infty\right\},\qquad \Lambda_L:=\{\lambda_l>0:\lambda_{l}^2<L,\text{ and }\lambda_{l}^2\in\N \}.$$
	Thus, $B_L$ is a finite-dimensional Banach space with respect to the norm $\mathcal{V}_{\rr}$. Since the constant appearing below does not depend on $L$, applying \Cref{prop:transference banach valued} for $p>1$ along with \cref{continuous variation} for $0<\lambda_{l}^2<L$ and thereafter taking $L\rightarrow\infty$, we obtain the following estimate:
	\begin{align*}
		\left\|t\left(\mathcal{J}_{t}(\mathfrak{B}_{1}f)\right)^{1/\rr}\right\|_{\ell^{p}}
		&\lesssim\left\|V_{\rr}\left(T_{\lambda_l}(\mathcal{U}f):\lambda_l\geq\lambda_{\lceil N^\epsilon\rceil}\right)\right\|_{\ell^{p}} \\
		&\lesssim\left\|\mathcal{V}_{\rr}\left(T_{\lambda_l}(\mathcal{U}f):\lambda_l\geq\lambda_{\lceil N^\epsilon\rceil}\right)\right\|_{\ell^{p}} \\ 
		&\lesssim \|\mathcal{U}f\|_{\ell^p}
	\end{align*}
    for all $\rr>2$ and $1<p<\infty$. Now to proceed further, we require the following lemma concerning the $\ell^p$ boundedness of the convolution operator associated to the multiplier $u(\xi)$, whose explicit proof is provided in \cite[Proposition 4.2]{LaceyLacunarySpherical}
	\begin{lemma}\cite[Proposition 4.2]{LaceyLacunarySpherical}\label{lem:lp bound for u}
		Let $\mathcal{U}$ be the operator defined as above. Then, for all $1\leq p\leq 2$ we have $\|\mathcal{U}f\|_{\ell^p}\lesssim \|f\|_{\ell^p}$.
	\end{lemma}
	To prove this result, the authors establish that the $\ell^1$ norm of the kernel of the operator $\mathcal{U}$ is bounded by $1$. Consequently, we have $\|\mathcal{U}f\|_{\ell^\infty}\lesssim \|f\|_{\ell^\infty}$. Interpolating this estimate with the $\ell^1$ boundedness, we obtain $\|\mathcal{U}f\|_{\ell^p}\lesssim \|f\|_{\ell^p}$ for all $p\geq 1$. Finally, applying this lemma we get the desired estimate that 
	$$\left\|t\left(\mathcal{J}_{t}(\mathfrak{B}_{1}f)\right)^{1/\rr}\right\|_{\ell^{p}}\lesssim \|f\|_{\ell^p},$$
	for every $1<p<\infty$ and $\rr>2$.


    \subsection{The $\ell^2$- estimates:}
	\subsubsection{Estimate for the family of averages associated to low-frequency:} Note that, from \cite{MSW2002} we have the multiplier $a_{\lambda_l}-m_{\lambda_l}$ associated with the error term given by the convolution operator $E_{\lambda_l}$ satisfies
	\begin{eqnarray}\label{error term estimate}
		\sup\limits_{\xi}|a_{\lambda_l}(\xi)-m_{\lambda_l}(\xi)|\lesssim_{n}\lambda_{l}^{-(n-4)/2}.
	\end{eqnarray}
	Moreover, for any family of operators $\A$, the pointwise inequality $t\left(\mathcal{J}_{t}(\A f)(x)\right)^{1/2}\lesssim \left(\sum_{l>\lceil N^\epsilon\rceil}\left|A_{\lambda_l}f(x)\right|^2\right)^{1/2}$ holds. Therefore, applying \eqref{error term estimate} along with the above pointwise inequality, we deduce that, for any $\rr>2$,
	\begin{align*}
		\left\|t\left(\mathcal{J}_{t}(\mathcal{E}f)\right)^{1/\rr}\right\|_{\ell^{2}}
		&\leq \left\|\left(\sum_{l>\lceil N^\epsilon\rceil}\left|E_{\lambda_l}f\right|^2\right)^{1/2}\right\|_{\ell^2}\\
		&\lesssim \left(\sum_{l>\lceil N^\epsilon\rceil}\frac{1}{\lambda_{l}^{n-4}}\right)^{1/2}\left\|f\right\|_{\ell^2}
		\lesssim \left(\sum_{\mu_l>\mu_{\lceil N^\epsilon\rceil}}\frac{1}{2^{\mu_{l}(n-4)}}\right)^{1/2}\left\|f\right\|_{\ell^2}
		\lesssim \frac{1}{N^{n-4}}\left\|f\right\|_{\ell^2}.
	\end{align*}
	\subsubsection{Estimate for the family $\mathcal{M}f$:} Note from the beginning that we restrict attention to those $N>N_0$ for which $\mu_{\lceil N^\epsilon\rceil}>N^2$. Here we have $l>\lceil N^\epsilon\rceil\geq N^\epsilon$ i.e., $l^{1/\epsilon}>N>N_0$ which implies $\mu_{l^{\frac{1}{\epsilon}\times\epsilon}}=\mu_l>l^{2/\epsilon}$. Hence, for all $l>\lceil N^\epsilon\rceil$ we obtain $\mu_{l}^{-1}<l^{-\frac{2}{\epsilon}}<l^{-2}$, since for small $\epsilon$, the quantity $1/\epsilon$ is a very large number. Combining this above observation with the fact that the $t$-jump is dominated by the associated square function, for all $\rr>2$ and the Gauss sum $G_q(a,b)$ is bounded by $q^{-n/2}$ for all $a\in U_q$, we deduce the following:
	\begin{align*}
		\left\|t\left(\mathcal{J}_{t}(\mathcal{M}f)\right)^{1/\rr}\right\|_{\ell^{2}}
		&\leq \left(\sum_{l>\lceil N^\epsilon\rceil}\left\|M_{2,\lambda_l}f\right\|_{\ell^2}^2\right)^{1/2}\\
		&\lesssim \left(\sum_{l>\lceil N^\epsilon\rceil}\left[\sum_{q>\mu_l}\frac{1}{q^{\frac{n}{2}-1}}\right]^2\right)^{1/2}\left\|f\right\|_{\ell^2}\\
		&\leq \left(\sum_{l>\lceil N^\epsilon\rceil}\frac{1}{l^{\frac{2}{\epsilon}(n-4)}}\right)^{1/2}\left\|f\right\|_{\ell^2}
		\leq \frac{1}{N^{\epsilon[\frac{2}{\epsilon}(n-4)-1]\times\frac{1}{2}}}\left\|f\right\|_{\ell^2}
		\lesssim N^{-(n-\frac{9}{2})}\left\|f\right\|_{\ell^2}.
	\end{align*}

	\subsubsection{Estimate for the family $\mathscr{M}f$:} In this case, by comparing the $t$-jump of the family $\mathscr{M}f$ with the square function associated to the family $\mathcal{M}f$, we arrive at
	\begin{align*}
		\left\|t\left(\mathcal{J}_{t}(\mathscr{M}f)\right)^{1/\rr}\right\|_{\ell^{2}}
		&\lesssim_n \left(\sum_{l>\lceil N^\epsilon\rceil}\left\|\mathfrak{M}_{2,\lambda_l}f\right\|_{\ell^2}^2\right)^{1/2}
		\leq \left(\sum_{l>\lceil N^\epsilon\rceil}\left\|M_{2,\lambda_l}f\right\|_{\ell^2}^2\right)^{1/2}
		\leq N^{-(n-\frac{9}{2})}\left\|f\right\|_{\ell^2},\;\text{for every $\rr>2$}.
	\end{align*}
	\subsubsection{Estimate for the family $\mathcal{S}f$:}We have the multiplier of the operator $S_{2,\lambda_l}$ is
	\begin{align*}
		s_{2,\lambda_l}(\xi)
		&=\sum_{q\nmid N!}\sum_{a\in U_{q}}\sum_{b\in \mathbb{Z}_{q}^{n}}G_{q}(a,b)\mathcal{F}_{\R^n}(\zeta_{q})\left(\xi-\frac{b}{q}\right)\mathcal{F}_{\R^n}(d\sigma_{\lambda_l})\left(\xi-\frac{b}{q}\right).
	\end{align*}
	Observe that, in this case, there is no scope of obtaining any decay in terms of the scale $\lambda_l$. Hence, we cannot use the square-function estimate to dominate the $t$-jump of the family $\mathcal{S}f$. Nevertheless, we can still control the $t$-jump of this family by its associated $\rr$-variation norm and apply \Cref{prop:transference banach valued} for $p=2$. Now, to estimate the $\rr$-variation norm associated with the family $\mathcal{S}f$, we first fix $a,\;\text{and}\;q$, and extract the sum over these indices, and thereafter we apply the transference principle as follows: 
	\begin{align*}
		&\left(\sum_{l}\left|S_{2,\lambda_{l+1}}f(x)-S_{2,\lambda_l}f(x)\right|^{\rr}\right)^{1/\rr}\\
		&\leq \sum_{q\nmid N!}\sum_{a\in U_{q}}\left(\sum_{l}\left|\mathcal{F}_{\Z^n}^{-1}\left(\sum_{b\in \mathbb{Z}_{q}^{n}}G_{q}(a,b)\mathcal{F}_{\R^n}(\zeta_{q})\left(\cdot-\frac{b}{q}\right)\mathcal{F}_{\R^n}\left(d\sigma_{\lambda_{l+1}}-d\sigma_{\lambda_l}\right)\left(\cdot-\frac{b}{q}\right)\mathcal{F}_{\Z^n}(f)(\cdot)\right)(x)\right|^{\rr}\right)^{1/\rr}\\
		&:=\sum_{q\nmid N!}\sum_{a\in U_{q}}\left(\sum_{l}\left|B_{\lambda_{l+1},q}f(x)-B_{\lambda_l,q}f(x)\right|^{\rr}\right)^{1/\rr}.
	\end{align*}
	Applying supremum over any finitely increasing sequence $\{\lambda_l\}$ in both sides of the above inequality, we obtain that
	\begin{align*}
		 V_r\left(S_{2,\lambda_l}f:\lambda_l\geq\lambda_{\lceil N^\epsilon\rceil}\right)(x)
		&=\left(\sup_{\lambda_1<\dots<\lambda_L}\sum_{l=1}^L\left|S_{2,\lambda_{l+1}}f(x)-S_{2,\lambda_l}f(x)\right|^{\rr}\right)^{1/\rr}\\
		&\leq \sum_{q\nmid N!}\sum_{a\in U_{q}}\left(\sup_{\lambda_1<\dots<\lambda_L}\sum_{l=1}^L\left|B_{\lambda_{l+1},q}f(x)-B_{\lambda_l,q}f(x)\right|^{\rr}\right)^{1/\rr}\\
		&:=\sum_{q\nmid N!}\sum_{a\in U_{q}}V_{\rr}\left(B_{\lambda_l,q}f:\lambda_l\geq\lambda_{\lceil N^\epsilon\rceil}\right)(x).
	\end{align*}
	Next, we write 
	\begin{align*}
		\sum_{b\in \mathbb{Z}_{q}^{n}}G_{q}(a,b)\mathcal{F}_{\R^n}(\zeta_{q})\left(\xi-\frac{b}{q}\right)\mathcal{F}_{\R^n}\left(d\sigma_{\lambda_l}\right)\left(\xi-\frac{b}{q}\right)\mathcal{F}_{\Z^n}(f)(\xi)
		=\mathfrak{t}_{\lambda_l}(\xi)\mathfrak{u}(\xi),
	\end{align*}
	where
    \begin{align*}
        \mathfrak{t}_{\lambda_l}(\xi)
        &=\sum_{b\in \mathbb{Z}_{q}^{n}}\mathcal{F}_{\R^n}(\zeta_{2q})\left(\xi-\frac{b}{q}\right)\mathcal{F}_{\R^n}\left(d\sigma_{\lambda_l}\right)\left(\xi-\frac{b}{q}\right),\;\text{and}\\
        \mathfrak{u}(\xi)
        &=\sum_{b\in \mathbb{Z}_{q}^{n}}G_{q}(a,b)\mathcal{F}_{\R^n}(\zeta_{q})\left(\xi-\frac{b}{q}\right)\mathcal{F}_{\Z^n}(f)(\xi).
    \end{align*}
     Therefore, we have
	\begin{align*}	
		V_{\rr}\left(B_{\lambda_l,q}f:\lambda_l\geq\lambda_{\lceil N^\epsilon\rceil}\right)(x)
		=V_{\rr}\left(\mathfrak{T}_{\lambda_l}(\mathbf{U}f):\lambda_l\geq\lambda_{\lceil N^\epsilon\rceil}\right)(x),
	\end{align*}
	where $\mathfrak{T}_{\lambda_l}$ and $\mathbf{U}$ are the convolution operator associated with the multipliers $\mathfrak{t}_{\lambda_l}(\xi)$ and $\mathfrak{u}(\xi)$ respectively. Thus, applying \Cref{prop:transference banach valued} for $p=2$ along with \cref{continuous variation} for $0<\lambda_{l}^2<L$ and thereafter taking $L\rightarrow\infty$, we obtain the following estimate:
	\begin{align*}
		\left\|t\left(\mathcal{J}_{t}(\mathcal{S}f)\right)^{1/\rr}\right\|_{\ell^{2}}
		\leq\|V_{\rr}\left(S_{2,\lambda_l}f:\lambda_l\geq\lambda_{\lceil N^\epsilon\rceil}\right)\|_{\ell^{2}}
		&\leq \sum_{q\nmid N!}\sum_{a\in U_{q}}\|\mathcal{V}_{\rr}\left(B_{\lambda_l,q}f:\lambda_l\geq\lambda_{\lceil N^\epsilon\rceil}\right)\|_{\ell^{2}}\\
		&= \sum_{q\nmid N!}\sum_{a\in U_{q}}\|\mathcal{V}_{\rr}\left(T_{\lambda_l}(\mathcal{U}f):\lambda_l\geq\lambda_{\lceil N^\epsilon\rceil}\right)\|_{\ell^{2}}\\
		&\lesssim \sum_{q\nmid N!}\sum_{a\in U_{q}}\|\mathcal{U}f\|_{\ell^{2}}\\
		&\leq\sum_{q>N}\frac{1}{q^{\frac{n}{2}-1}}\|f\|_{\ell^2}
		\leq N^{-(\frac{n}{2}-2)}\|f\|_{\ell^2},\quad\text{for all $\rr>2$}.
	\end{align*}
	

     \subsubsection{Estimate for the families $\mathfrak{B}_{2}f$:} We have 
	\begin{align*}
		\mathfrak{b}_{2,\lambda_l}(\xi)
		&=\sum_{a\in \Z_{N!}}\sum_{b\in \Z_{N!}^{n}}G_{N!}(a,b)\left\{\widehat{\zeta_{\frac{N!}{\nu}}}\left(\xi-\frac{b}{N!}\right)-\widehat{\zeta_{2N!}}\left(\xi-\frac{b}{N!}\right)\right\}\widehat{d\sigma}\left(\lambda_l\left(\xi-\frac{b}{N!}\right)\right).
	\end{align*}
	Observe that, this multiplier is nonzero for $\frac{1}{2\times N!}\lesssim\left|\xi-\frac{b}{N!}\right|\leq\frac{\nu}{N!}$. Moreover, since for $l>\lceil N^\epsilon\rceil$ one has $\mu_l\geq \mu_{\lceil N^\epsilon\rceil}>N^2$, it follows that
	\begin{align*}
		\left\|t\left(\mathcal{J}_{t}(\mathfrak{B}_{2}f)\right)^{1/\rr}\right\|_{\ell^{2}}
		&\lesssim_n \left(\sum_{l>\lceil N^\epsilon\rceil}\left\|\mathscr{B}_{2,\lambda_l}f\right\|_{\ell^2}^2\right)^{1/2}\\
		&\leq \left(\sum_{l>\lceil N^\epsilon\rceil}\left|\left(\frac{N!}{\lambda_l}\right)^{\frac{n-1}{2}}\right|^2\right)^{1/2}\|f\|_{\ell^2}\\
		&\leq \left(\sum_{\mu_l>\mu_{\lceil N^\epsilon\rceil}}\left(\frac{N!}{\mu_l!}\right)^{n-1}\right)^{1/2}\|f\|_{\ell^2}\\
		&\leq \left(\sum_{\mu_l>\mu_{\lceil N^\epsilon\rceil}}\left(\frac{N!}{(\mu_l!)^{1/2}}\right)^{n-1}\left(\frac{1}{2^{\mu_l}}\right)^{\frac{n-1}{2}}\right)^{1/2}\|f\|_{\ell^2}\\
		&\leq \left(\frac{1}{2^{\mu_{\lceil N^\epsilon\rceil}}}\right)^{\frac{n-1}{4}}\|f\|_{\ell^2}
        \leq N^{-(n-1)}\|f\|_{\ell^2},\quad\text{for every $\rr>2$}.
	\end{align*}
    Here, in the above estimate, we have used the fact that, for any fixed $b\in\Z^{n}_{N!}$, the quantity $$\tilde{G}:=(N!)^{-n}\sum\limits_{a\in\Z_{N!}}\sum\limits_{m\in\Z^{n}_{N!}}e\left(\frac{a}{N!}|m|^2+\frac{b\cdot m}{N!}\right)$$ is bounded by $1$. To see this, observe that
	\begin{align}
		\tilde{G}
		=(N!)^{-n}\sum_{a\in\Z_{N!}}\sum_{m\in\Z^{n}_{N!}}e\left(\frac{a}{N!}|m|^2+\frac{b\cdot m}{N!}\right)
		=\frac{1}{(N!)^n}\sum_{d|N!}\sum_{a\in U_{d}}\sum_{m\in\Z^{n}_{N!}}e\left(\frac{a}{d}|m|^2+\frac{b\cdot m}{N!}\right).\label{twisted Gauss sum}
	\end{align}
	Next, consider $S:=\sum\limits_{m\in\Z_{N!}}e\left(\frac{a}{d}|m|^2+\frac{b\cdot m}{N!}\right)$. Therefore,
	\begin{align*}
		|S|^2
		&=\sum_{n\in\Z_{N!}}\sum_{m\in\Z_{N!}}e\left(\frac{a}{d}(m^2-n^2)+\frac{b\cdot (m-n)}{N!}\right)\\
		&=\sum_{0\leq n\leq N!-1}\sum_{0\leq m\leq N!-1}e\left(\frac{a}{d}(m^2-n^2)+\frac{b\cdot (m-n)}{N!}\right)\\
		&=\sum_{1\leq n\leq N!}\sum_{1-n\leq h\leq N!-n}e\left(\frac{a}{d}((n+h)^2-n^2)+\frac{b\cdot ((n+h)-n)}{N!}\right)\\
		&=\sum_{1\leq h\leq N!}\left[\sum_{1-h\leq n\leq N!-h}e\left(\frac{2nha}{d}\right)\right]e\left(\frac{ah^2}{d}+\frac{b\cdot h}{N!}\right).
	\end{align*}
	Notice that, $\sum\limits_{1-h\leq n\leq N!-h}e\left(\frac{2nha}{d}\right)=\frac{N!}{d}\sum\limits_{1-h\leq n\leq d-h}e\left(\frac{2nha}{d}\right)$. Invoking this in the above sum, we obtain that
	\begin{align*}
		&\sum_{1\leq h\leq N!}\left[\sum_{1-h\leq n\leq N!-h}e\left(\frac{2nha}{d}\right)\right]e\left(\frac{ah^2}{d}+\frac{b\cdot h}{N!}\right)\\
		&=\frac{N!}{d}\sum_{1\leq h\leq N!}\left[\sum_{1-h\leq n\leq d-h}e\left(\frac{2nha}{d}\right)\right]e\left(\frac{ah^2}{d}+\frac{b\cdot h}{N!}\right)\\
		&=\frac{N!}{d}\sum_{k=1}^{N!/d}\left\{\sum_{(k-1)d< h\leq dk}\left[\sum_{1-h\leq n\leq d-h}e\left(\frac{2nha}{d}\right)\right]e\left(\frac{ah^2}{d}+\frac{b\cdot h}{N!}\right)\right\}.
	\end{align*}
	Since $a\in U_d$, the inner sum is nonzero only when $2h\equiv 0\mod d$; i.e., $d\mid 2h$. Now we proceed to analyze the bound of the above term through a casewise argument.
	\begin{enumerate}[(i)]
		\item Case-$1$($d$ is odd integer): In this case, we have $d\mid h$. Note that, within each interval $((k-1)d,kd]$, there is exactly one choice of $h$ which is divisible by $d$, namely $h=kd$. Consequently, in this situation the quantity $|S|^2$ is bounded by $\left(\frac{N!}{d}\right)^2\times d$.
		\item Case-$2$($d$ is even integer): Observe that, in this case, each interval $((k-1)d,kd]$ contains exactly two choices for $h$ divisible by $d$, namely $kd$ and $(k-1)d+\frac{d}{2}$. Thus, in this case the quantity $|S|^2$ is bounded by $\left(\frac{N!}{d}\right)^2\times d$.
	\end{enumerate}
	
	Applying this estimate in \eqref{twisted Gauss sum}, we obtain that
	\begin{align*}
		|\tilde{G}|
		&\lesssim \frac{1}{(N!)^n}\sum_{d|N!}\sum_{a\in U_{d}}\left(\frac{N!}{d}\right)^n\times d^{\frac{n}{2}} \lesssim 1.
	\end{align*}
	This concludes the proof of \Cref{thm:varSphereSparse}.
	 
\section{Ergodic Transference Principle: Proof of \Cref{thm:variational transference principle}}\label{sec:ergodic transference}
		Let us fix $f\in L^p(X),~\eta>0,~\Lambda\in\N$, and for every $x\in X$ define the function 
		$$
    	\gamma_{x}(y)=\left\{\begin{array}{cc}
    	f\left(T_1^{y_1}\circ T_2^{y_2}\circ\cdots\circ T_n^{y_n}x\right), & \text { if } \|y\|_{\ell^{\infty}} \leq\Lambda^{\frac{1}{d}}\left(1+\frac{\eta}{n}\right), \\~\\
    	0, & \text { otherwise. }
    	\end{array}\right.
    	$$
		Also, for every positive integer $\lambda$, define the set 
		$G_{\lambda}=\left\{x\in\Z^n:\|x\|_{\ell^{\infty}}\leq \lambda^{\frac{1}{d}}\right\}.$
		For each $m\in G_{\Lambda}$ and $y\in G_{\lambda}$ with $\lambda^{\frac{1}{d}}<\Lambda^{\frac{1}{d}}\eta/n$, since we have $\|m-y\|_{\ell^{\infty}}\leq\Lambda^{\frac{1}{d}}\left(1+\frac{\eta}{n}\right)$, the following equality holds:
		\begin{align}
			\mathbb{A}_{\lambda}^{\Rr}f(T_1^{m_1}\circ T_2^{m_2}\circ\cdots\circ T_n^{m_n}x)=\frac{1}{r_{\Rr,\phi}(\lambda)}\sum_{y:\Rr(y)=\lambda}\phi\left(\frac{y}{\lambda^{\frac{1}{d}}}\right)\gamma_x(m-y)=A_{\lambda}^{\Rr}\gamma_x(m)\label{relation between Birch-Magyar avg. with its ergodic analg.}.
		\end{align}
		Now, \cref{variation for BM} along with \cref{relation between Birch-Magyar avg. with its ergodic analg.} we get
		\begin{align*}
			&\sum_{m\in G_{\Lambda}}|V_r(\mathbb{A}_{\lambda}^{\Rr}f(T_1^{m_1}\circ T_2^{m_2}\circ\cdots\circ T_n^{m_n}x):\lambda\in D\cap(0,\Lambda^{\frac{1}{d}}\eta/n))|^p\\
			&\leq \sum_{m\in G_{\Lambda}}|V_r(A_{\lambda}^{\Rr}\gamma_x(m):\lambda\in D\cap(0,\Lambda^{\frac{1}{d}}\eta/n))|^p\\
			&\leq \|V_r(A_{\lambda}^{\Rr}\gamma_x(m):\lambda\in D)\|_{\ell^p(\Z^n)}^p\lesssim \|\gamma_x\|_{\ell^p(\Z^n)}^p.
		\end{align*} 
		Therefore,
		\begin{align*}
			&\sum_{m\in G_{\Lambda}}\int_{x\in X}|V_r(\mathbb{A}_{\lambda}^{\Rr}f(T_1^{m_1}\circ T_2^{m_2}\circ\cdots\circ T_n^{m_n}x):\lambda\in D\cap(0,\Lambda^{\frac{1}{d}}\eta/n))|^pd\mu(x)\\
			&\lesssim \sum_{m\in G_{\Lambda\left(1+\frac{\eta}{n}\right)^d}}\int_{x\in X}|\gamma_{x}(m)|^pd\mu(x)\\
			&=\sum_{m\in G_{\Lambda\left(1+\frac{\eta}{n}\right)^d}}\int_{x\in X}|f(T_1^{m_1}\circ T_2^{m_2}\circ\cdots\circ T_n^{m_n}x)|^pd\mu(x).
		\end{align*}
		For each $i$, since $T_{i}^{m_i}$ is a measure preserving transformation, we get from the above inequality that
		\begin{align*}
			\sum_{m\in G_{\Lambda}}\|V_r(\mathbb{A}_{\lambda}^{\Rr}f:\lambda\in D\cap(0,\Lambda^{\frac{1}{d}}\eta/n))\|_{L^p(X)}^p
			&\lesssim \sum_{m\in G_{\Lambda\left(1+\frac{\eta}{n}\right)^d}}\|f\|_{L^p(X)}^p,
		\end{align*}
		which implies
		\begin{align*}
			\|V_r(\mathbb{A}_{\lambda,\phi}^{\Rr}f:\lambda\in D\cap(0,\Lambda^{\frac{1}{d}}\eta/n))\|_{L^p(X)}^p
			&\lesssim \frac{1}{(2\Lambda)^{\frac{n}{d}}}\left(2\Lambda^{\frac{1}{d}}\left(1+\frac{\eta}{n}\right)\right)^n\|f\|_{L^p(X)}^p.
		\end{align*}
		Applying $\Lambda\rightarrow\infty$ in the above inequality together with the monotone convergence theorem, we obtain that
		\begin{align*}
			\|V_r(\mathbb{A}_{\lambda,\phi}^{\Rr}f:\lambda\in D)\|_{L^p(X)}^p
			&\lesssim\left(1+\frac{\eta}{n}\right)^n\|f\|_{L^p(X)}^p.
		\end{align*}
		Now, taking $\eta\rightarrow 0^{+}$, we get \cref{variation for ergodic BM}, which completes the proof of \Cref{thm:variational transference principle}.\qed	

\section{Obstructions to the $\ell^p$ boundedness of $r$-variation for discrete spherical averages}\label{sec:counterexample}
 This section is devoted to the counterexamples concerning the $\ell^p$ boundedness of $r$-variations associated with the discrete spherical averages. The first example is the binary parity obstruction. 
	\begin{example}\label{example1}
		Let us consider the sequence of radii $\{\lambda_j\}_{j\geq 1}=\{4^j+\textbf{1}_{\{\text{j is odd}\}}\}_{j\geq 1}$. We set $f(x):=(-1)^{x_1+\dots+x_n}\textbf{1}_{[-\lambda_M,\lambda_M]^n}(x)$ and we have
		\begin{align*}
            A_{\lambda_j}f(x) 
            &=\frac{1}{r_n(\lambda_j)} \sum_{\substack{y\in\mathbb Z^n:|y|=\lambda_j}} f(x+y)\\ 
			&= \frac{1}{r_n(\lambda_j)} \sum_{\substack{y\in\mathbb Z^n:|y|=\lambda_j}} (-1)^{x_1+\dots+x_n+y_1+\dots+y_n}\textbf{1}_{[-\lambda_M,\lambda_M]^n}(x+y).
        \end{align*}
		For $y\in\mathbb{S}^{n-1}_{\lambda_j}\cap\Z^n$, we have $\|y\|_{\infty}\leq\lambda_j$. Let $M$ be a sufficiently large natural number. Now for $1\leq j\leq M-2$, $y\in\mathbb{S}^{n-1}_{\lambda_j}\cap\Z^n$ and $\|x\|_{\infty}\leq 4^{M-2}$ we have $\|x+y\|_{\infty}\leq 2\cdot 4^{M-2}+1\leq\lambda_M$. Therefore, we obtain that
		\begin{align*}
            A_{\lambda_j}f(x) 
            &= \frac{f(x)}{r_n(\lambda_j)} \sum_{\substack{y\in\mathbb Z^n:|y|=\lambda_j}} (-1)^{y_1+\dots+y_n}
			=(-1)^{\lambda_{j}^2},\quad\text{for all $x\in\Z^n$ with }\|x\|_{\infty}\leq 4^{M-2}.
        \end{align*}
        Here, we derive the final equality by incorporating the fact that the product of any two consecutive integers is even, which further implies that $(-1)^{y_{j}}=(-1)^{y_{j}^2}$. Observe that, $|A_{\lambda_{j+1}}f(x)-A_{\lambda_j}f(x)|=2$ for all $1\leq j\leq M-3$ and $x\in\Z^n$ with $\|x\|_{\infty}\leq 4^{M-2}$. Thus, we have
		\begin{align*}
			V_r\left(A_{\lambda_j}f:j\geq 1\right)(x)
			&\geq\left(\sum_{j=1}^{M-3}2^r\right)^{1/r}=2(M-3)^{1/r}.
		\end{align*}
		Since $\|f\|_{\ell^p}\lesssim4^{Mn/p}$, we deduce that
		\begin{align*}
			\|V_r\left(A_{\lambda_j}f:j\geq 1\right)\|_{\ell^p}
			\geq\left(\sum_{\substack{x\in\Z^n:\\ \|x\|_{\infty}\leq 4^{M-2}}}2(M-3)^{p/r}\right)^{1/p}
			\gtrsim2(M-3)^{1/r}4^{Mn/p}
			\gtrsim(M-3)^{1/r}\|f\|_{\ell^p}.
		\end{align*}
		Hence, letting $M\rightarrow\infty$ implies that the norm $\|V_r\left(A_{\lambda_j}:j\geq 1\right)\|_{\ell^p\rightarrow\ell^p}$ is not finite for any fixed $p,r\geq1$. 
		\end{example}
    \subsection{Proofs of \Cref{prop:variationunbounded} and \Cref{thm:variation unbounded}}\label{proof:variationunbounded}
	Since we know that the pointwise convergence of ergodic averages is a consequence of the $\ell^p$ boundedness of $r$-variation associated to this family of averages, to prove \Cref{thm:variation unbounded}, it suffices to construct an example showing that discrete spherical averages fail to converge pointwise along the sequence of radii $\{2^j\}_{j\geq 0}$ in general dynamical system $(X,\mathcal{B}(X), \mu, T)$. Moreover, \Cref{prop:variationunbounded} is a direct consequence of \Cref{thm:variation unbounded} by the transference principle \Cref{thm:variational transference principle}.
	
	To construct such a counterexample, let us consider $X=\Z/5\Z $, $\mathcal{B}(X)$ be the $\sigma$-algebra of all subsets of $X$, $ T_1(x)=(x+1)\mod 5$ and $T_i(x)=x\mod 5$ for $2\leq i\leq 5$ and $f(x):=w^x=e^{\frac{2\pi ix}{5}}$. Choose $(2^{j})^2=N_j$. Then, we have
	\begin{align*}
		A_{2^j}f(x)
		=\frac{1}{r_5(2^j)}\sum_{\substack{y\in\Z^5:\\|y|=2^j}}f(T^{y_1}_1\dots T^{y_5}_5(x))
        &=\frac{1}{r_5(\sqrt{N_j})}\sum_{\substack{y\in\Z^5:\\|y|^2=N_j}}f(x+y_1)=\frac{1}{r_5(\sqrt{N_j})}\sum_{\substack{y\in\Z^5:|y|^2=N_j}}w^{x+y_1}.
	\end{align*}	
	Since the conditions $|y|^2=N_j$ and $y_i\equiv u_i\mod 5$ together imply $|u|^2\equiv N_j\mod 5$, we have 
    \begin{align*}
        A_{2^j}f(x)
        &=\frac{f(x)}{r_5(\sqrt{N_j})}\sum_{\substack{u\in\left(\Z/5\Z\right)^5:\\|u|^2\equiv N_j\mod 5}}w^{u_1}|\{y\in\Z^5:|y|^2=N_j,\; y_i\equiv u_i\mod 5, \forall \; i=1,\ldots,5\}|.
    \end{align*}
    Now define 
	\begin{align*}
		R^{u}_n(N_j):=|\{y\in\Z^n:|y|^2=N_j,\; y_i\equiv u_i\mod 5, \forall \; i=1,\ldots,n\}|.
	\end{align*}
	Thus,
	\begin{align*}
		A_{2^j}f(x)
		&=f(x)\sum_{\substack{u\in\left(\Z/5\Z\right)^5:\\|u|^2\equiv N_j\mod 5}}w^{u_1}\frac{R^{u}_5(N_j)}{r_5(\sqrt{N_j})}.
	\end{align*}
	Next, we study the pointwise convergence of these averages along the two subsequences $\{2^{2j}\}_{j\geq 0}$ and $\{2^{2j+1}\}_{j\geq 0}$ of the sequence of radii $\{2^j\}_{j\geq 0}$. If we show that the averages along these two subsequences converge to different limits, then we can conclude that the averages along the sequence of radii $\{2^j\}_{j\geq 0}$ fail to converge pointwise. Note the following asymptotic formula for $r_5(\sqrt{N_j})$, given in \cite[Theorem 20.9, pp.472]{AnalyticNumberTheoryBOOK}
	\begin{align*}
		r_5(\sqrt{N_j})=\frac{(2\pi)^{\frac{5}{2}}N^{\frac{3}{2}}_j}{\Gamma\left(\frac{5}{2}\right)2^{\frac{5}{2}}}\mathfrak{S}(N_j)+O\left(N^{1+\epsilon}_j\right), \quad \text{as } N_j\to\infty,
	\end{align*}
	where $\mathfrak{S}(N_j)$ denotes the singular series, which can be written as the product of local densities of $p$-adic solutions to $|y|^2=N_j$ and satisfies the bound $C_1\leq\mathfrak{S}(N_j)\leq C_2$ for some constants $C_1,C_2$ independent of $N_j$. We claim that the lattice points lying on a sphere of fixed radii are equidistributed over all congruence classes mod $5$. While, this appears to be a well known result in the folklore, we could not find a precise reference in the literature. Thus, we provide a simple and self-contained proof of this fact, which is only based on a multiplier bound in the work of \cite{MSW2002}. Let us first define 
	\[C_a := \Big\{\, v \in (\Z/5\Z)^5 : |v|^2 \equiv a \pmod{5} \,\Big\}.\]
	Now, we have equidistribution result stated below.
	\begin{lemma}[Equidistribution result]\label{lem: equidistribution result}
    Let $N_j \in \N$, and write $N_j \equiv a \pmod{5}$ with $a \in \{1,4\}$. For $v \in (\Z/5\Z)^5$, let $R_v(N_j) := |\big\{\, y \in \Z^5 : |y|^2 = N_j,\; y \equiv v \pmod{5} \,\big\}|$.
   Then for every $v \in C_a$,
   \begin{align}
	R_v(N_j) = \frac{r_5(\sqrt{N_j})}{|C_a|} + O\left(r_5(\sqrt{N_j})^{1-\delta}\right),\quad\text{for an absolute constant $\delta > 0$.}\label{relation between r and R}
   \end{align}	
	\end{lemma}
	To proceed, we also require the following lemma:
	\begin{lemma}\label{lem: counting A_N}
		The cardinality of the set $C_a$ is $650$ for all $a\in\{1,4\}$. Moreover, for 
		\begin{align*}
			C_a(t):=\{u=(u_1,\dots,u_5)\in C_a:u_1=t\},\quad\text{we have}\;\;|C_a(t)|=
			\begin{cases}
				145 &\text{if}\;\; t^2\equiv a\mod 5,\\
				120 &\text{if}\;\; t^2\not\equiv a\mod 5.
			\end{cases}
		\end{align*}
	\end{lemma}
	For the moment, let us assume \Cref{lem: equidistribution result} and \Cref{lem: counting A_N} and complete the proof of the counterexample; subsequently, we will establish their proofs. Next, observe that $4^{2j}\equiv 1 \pmod{5}$ and $4^{2j+1}\equiv 4 \pmod{5}$ for all $j\geq 0$, thus using \Cref{lem: equidistribution result}, we have 
	\begin{align*}
		A_{2^{2j}}f(x)
		&=f(x)\sum_{\substack{u\in C_1}}w^{u_1}\frac{R^{u}_5(4^{2j})}{r_5(2^{2j})}\\
		&=f(x)\sum_{\substack{u\in C_1}}w^{u_1}\left(\frac{1}{|C_1|}+O\left(r_5(2^{2j})^{-\delta}\right)\right)\\
		&=f(x)\sum_{t=0}^4 w^{t}\left|\left\{u=(u_1,\dots,u_5)\in C_1: u_1 =t\right\}\right|\left(\frac{1}{|C_1|}+O\left(r_5(2^{2j})^{-\delta}\right)\right)\\
		&=f(x)\sum_{t=0}^4 w^{t}|C_1(t)|\left(\frac{1}{|C_1|}+O\left(r_5(2^{2j})^{-\delta}\right)\right).
	\end{align*}
	Applying \Cref{lem: counting A_N}, we obtain that $C_1(0)=C_1(2)=C_1(3)=120$ and $C_1(1)=C_1(4)=145$, thus we have $A_{2^{2j}}f(x)=25f(x)(w+w^4)\left(\frac{1}{|C_1|}+O\left(r_5(2^{2j})^{-\delta}\right)\right).$
	Next, we have
	\begin{align*}
		A_{2^{2j+1}}f(x)
		&=f(x)\sum_{\substack{u\in C_4}}w^{u_1}\frac{R^{u}_5(4^{2j+1})}{r_5(2^{2j+1})}\\
		&=f(x)\sum_{\substack{u\in C_4}}w^{u_1}\left(\frac{1}{|C_4|}+O\left(r_5(2^{2j+1})^{-\delta}\right)\right)\\
		&=f(x)\sum_{t=0}^4 w^{t}\left|\left\{u=(u_1,\dots,u_5)\in C_4:u_1 =t\right\}\right|\left(\frac{1}{|C_4|}+O\left(r_5(2^{2j+1})^{-\delta}\right)\right)\\
		&=f(x)\sum_{t=0}^4 w^{t}|C_4(t)|\left(\frac{1}{|C_4|}+O\left(r_5(2^{2j+1})^{-\delta}\right)\right).
	\end{align*}
	Similarly, using \Cref{lem: counting A_N}, we obtain that $C_4(0)=C_4(1)=C_4(4)=120$ and $C_4(2)=C_4(3)=145$, thus we have $A_{2^{2j+1}}f(x)=25f(x)(w^2+w^3)\left(\frac{1}{|C_4|}+O\left(r_5(2^{2j+1})^{-\delta}\right)\right).$ Hence, the two sequences $\{A_{2^{2j}}f(x)\}_{j\geq0}$ and $\{A_{2^{2j+1}}f(x)\}_{j\geq0}$ converges to $\frac{f(x)(\sqrt{5}-1)}{52}$ and $\frac{f(x)(-\sqrt{5}-1)}{52}$ respectively.
	Therefore, we conclude that the averages along the two subsequences $\{2^{2j}\}_{j\geq 0}$ and $\{2^{2j+1}\}_{j\geq 0}$ converge to different limits. This completes the construction of the counterexample.

	\begin{proof}[\textbf{Proof of \Cref{lem: counting A_N}:}]
	We define
	\begin{align*}
		C_{a}^i := \Big\{\, v=(v_1,\dots,v_i) \in (\Z/5\Z)^i : v_{1}^2+\dots+v_{i}^2 \equiv a \pmod{5} \,\Big\}\quad\text{for }i=2\;\text{and}\;4.
	\end{align*}
	By observation, we have
	\begin{align}
		|C_{a}^2|=
		\begin{cases}
			9, &\text{if}\;\;a\equiv 0\mod 5,\\
			4, &\text{if}\;\;a\not\equiv 0\mod 5.
		\end{cases}\label{evaluating C_a^2}
	\end{align}
	Next, observe that $|C_{a}^4|=\sum\limits_{t\in\Z_5}|C_{t}^2||C_{a-t}^2|$. Therefore, by an application of \ref{evaluating C_a^2}, we obtain that $|C_{0}^4|=145$ and for $b\neq 0$ we have $|C_{b}^4|=120.$ Thus, to compute $|C_a|$, we first fix the first coordinate $u_1=t$. Hence, the remaining coordinates satisfy $v_{2}^2+\dots+v_{4}^2=a-t^2$. This yields, $|C_a|=\sum\limits_{t\in\Z_5}|C_{a-t^2}^4|$. Using this formula, along with the fact $1^2,4^2\equiv 1\mod 5$ and $2^2,3^2\equiv 4\mod 5$, we compute that $|C_1|=|C_4|=650$. Furthermore, we note that, for a fixed $u_1=t\in\Z_5$, we have
	\begin{align*}
		|C_a(t)|=|C_{a-t^2}^4|
		&=
		\begin{cases}
			|C_{0}^4| &\text{if}\;\; t^2\equiv a\mod 5,\\
			|C_{b}^4|,\;\text{for some $b\neq 0$} &\text{if}\;\; t^2\not\equiv a\mod 5,
		\end{cases}\\
		&=
		\begin{cases}
			145 &\text{if}\;\; t^2\equiv a\mod 5,\\
			120 &\text{if}\;\; t^2\not\equiv a\mod 5.
		\end{cases}
	\end{align*}
	This concludes the proof of \Cref{lem: counting A_N}.
	\end{proof} 
    \begin{proof}[\textbf{Proof of \Cref{lem: equidistribution result}:}]
	Let  $n = 5$  and $ N_j \in \N$  with $N_j\equiv a \pmod{5}$ , where $a\in\{1,4\}$. We begin by recalling the definition of the exponential sum
	\begin{align*}
		a_{\lambda_j}(\xi) = \frac{1}{r_5(\lambda_j)}\sum\limits_{y\in\Z^5:|y|=\lambda_j}e(-y\cdot\xi).
	\end{align*}
	In proving \Cref{lem: equidistribution result}, our approach of decomposing the above exponential sum relies on the work of Magyar, Stein, and Wainger \cite{MSW2002}, where Hardy-Littlewood-Ramanujan circle method is employed.
	Now, choosing $\lambda_j=\sqrt{N_j}$, for any $h\in(\Z/5\Z)^5$ we have
		\begin{align*}
			a_{\sqrt{N_j}}\left(\frac{h}{5}\right) = \frac{1}{r_5(\sqrt{N_j})}\sum\limits_{y\in\Z^5:|y|^2=N_j}e\left(-\frac{y}{5}\cdot h\right).
		\end{align*}
		Applying \cite[Proposition 4.1, pp.201]{MSW2002}, we obtain that 
		\begin{align}
			a_{\sqrt{N_j}}\left(\frac{h}{5}\right)
			=m_{\sqrt{N_j}}\left(\frac{h}{5}\right) +O\left(N^{-\frac{1}{4}}_j\right),\label{relation between a and m}
		\end{align}
		where $m$ is as defined in \eqref{main multiplier}. Next, to move forward, we require the following lemma:
		\begin{lemma}\label{lem: relation between m(h) and m(0)}
			Let, $m$ be as defined in \eqref{main multiplier} and $C_a :=\left\{\, u \in (\Z/5\Z)^5 : |u|^2 \equiv a \pmod{5} \,\right\}$.
			Then, for any $h\in(\Z/5\Z)^5$, we have
			\begin{align*}
				m_{\sqrt{N_j}}\left(\frac{h}{5}\right)
				=\frac{m_{\sqrt{N_j}}(0)}{|C_a|}\sum_{u\in C_a}e\left(-\frac{h\cdot u}{5}\right).
			\end{align*}
		\end{lemma}
		We delay the proof of this lemma to the end of this section. Now, combining this lemma with \eqref{relation between a and m} along with the fact $m_{\sqrt{N_j}}(0)=1+O(N^{-\frac{1}{4}}_j)$, we obtain that
		\begin{align}
			a_{\sqrt{N_j}}\left(\frac{h}{5}\right)
			&=\frac{m_{\sqrt{N_j}}(0)}{|C_a|}\sum_{u\in C_a}e\left(-\frac{h\cdot u}{5}\right)+O\left(N^{-\frac{1}{4}}_j\right)\nonumber\\
			&=\frac{1}{|C_a|}\sum_{u\in C_a}e\left(-\frac{h\cdot u}{5}\right)+O\left(N^{-\frac{1}{4}}_j\right).\label{another form of E_h}
		\end{align}
		Next, we express $R_v(N_j)$ in terms of the exponential sum, as follows:
		\begin{align*}
			R_v(N_j)
			&=|\{y\in\Z^5:|y|^2=N_j, y\equiv v\mod 5\}|\\
			&=\sum_{y\in\Z^5:|y|^2=N_j}\frac{1}{5^5}\sum_{h\in(\Z/5\Z)^5}e\left(\frac{(v-y)\cdot h}{5}\right)\\
			&=\frac{r_5(\sqrt{N_j})}{5^5}\sum_{h\in(\Z/5\Z)^5}a_{\sqrt{N_j}}\left(\frac{h}{5}\right)e\left(\frac{v\cdot h}{5}\right).
		\end{align*}
		Since for all $u,v\in C_a\subset(\Z/5\Z)^5$ we have $5^{-5}\sum_{h\in(\Z/5\Z)^5}e\left(\frac{(v-u)\cdot h}{5}\right)=1_{\{v=u\}}$, incorporating \eqref{another form of E_h} into the above equality along with the fact that $r_5(\sqrt{N_j})\approx N^{\frac{3}{2}}_j$, we obtain the following desired estimate:
		\begin{align*}
			R_v(N_j)
			&=\frac{r_5(\sqrt{N_j})}{5^5}\sum_{h\in(\Z/5\Z)^5}\left[\frac{1}{|C_a|}\sum_{u\in C_a}e\left(-\frac{h\cdot u}{5}\right)+O\left(N^{-\frac{1}{4}}_j\right)\right]e\left(\frac{v\cdot h}{5}\right)\\
			&=\frac{r_5(\sqrt{N_j})}{|C_a|}\sum_{u\in C_a}\left[\frac{1}{5^5}\sum_{h\in(\Z/5\Z)^5}e\left(\frac{(v-u)\cdot h}{5}\right)\right]+O\left(N^{\frac{5}{4}}_j\right)\\
			&=\frac{r_5(\sqrt{N_j})}{|C_a|}+O\left(N^{\frac{5}{4}}_j\right),
		\end{align*}
	Further, applying $r_5(\sqrt{N_j})\approx N^{\frac{3}{2}}_j$, the above expression can be written as $\frac{r_5(\sqrt{N_j})}{|C_a|}+O\left(r_5(\sqrt{N_j})^{\frac{5}{6}}\right)$.
	\end{proof}	
	\begin{proof}[\textbf{Proof of \Cref{lem: relation between m(h) and m(0)}:}]
		For any integers $q\geq 1,$ and $s$, let us define the Ramanujan sum $r_q(s)$ as $r_q(h):=\sum_{a\in U_q}e\left(\frac{ah}{q}\right)$. Also, we define the exponential sum $\mathscr{R}_q(h)$ as 
		\begin{align}
			\mathscr{R}_q(h):=
			\begin{cases}
				\frac{1}{q^5}\sum\limits_{x\in(\Z/q\Z)^5}e\left(h\cdot\frac{x}{5}\right)r_q(|x|^2-N), & \text{for } 5\mid q,\; h\neq 0,\\~\\
				\frac{1}{q^5}\sum\limits_{x\in(\Z/q\Z)^5}r_q(|x|^2-N), & \text{for }  h= 0,\;q\geq 1.
			\end{cases}\label{definition of R_q}
		\end{align}
		Notice that, choosing $\lambda_j=\sqrt{N_j}$ in \eqref{main multiplier}[pp.6] we have
		\begin{align*}
			m_{\sqrt{N_j}}\left(\frac{h}{5}\right)
			&=c_5\sum_{q=1}^{\infty}\sum_{a\in U_{q}}e_{q}(-N_ja)\sum_{b\in \mathbb{Z}_{q}^{n}}G_{q}(a,b)\mathcal{F}_{\R^n}(\zeta_{q})\left(\frac{h}{5}-\frac{b}{q}\right)\mathcal{F}_{\R^n}(d\sigma_{\sqrt{N}})\left(\frac{h}{5}-\frac{b}{q}\right).
		\end{align*}
		The support condition of $\mathcal{F}_{\R^n}(\zeta)$ implies that, the only choice of $b$ in the above sum is $b=\frac{hq}{5}$, it further forces that $q$ has to be multiple of $5$ while $h\neq 0$, unless the whole quantity will vanish. Therefore, we have
		\begin{align}
			m_{\sqrt{N_j}}\left(\frac{h}{5}\right)=
			\begin{cases}
				c_5\sum\limits_{q:5\mid q}\sum\limits_{a\in U_{q}}e_{q}(-N_ja)G_{q}\left(a,\frac{hq}{5}\right), & \text{for }  h\neq 0,\\~\\
				c_5\sum\limits_{q=1}^{\infty}\sum\limits_{a\in U_{q}}e_{q}(-N_ja)G_{q}\left(a,0\right), & \text{for }  h= 0.
			\end{cases}\label{m in terms of R}
		\end{align}
		Next, observe that, for any $q\geq 1$ and $h\in(\Z/5\Z)^5$, we have
		\begin{align*}
			\sum_{a\in U_{q}}e_{q}(-N_ja)G_{q}\left(a,\frac{hq}{5}\right)
			&=\sum_{a\in U_{q}}e_{q}(-N_ja)\frac{1}{q^5}\sum_{x\in(\Z/q\Z)^5}e_{q}\left(a|x|^2+\frac{hq}{5}\cdot x\right)\\
			&=\frac{1}{q^5}\sum_{x\in(\Z/q\Z)^5}e\left(h\cdot\frac{x}{5}\right)\sum_{a\in U_{q}}e_{q}(a(|x|^2-N_j))\\
			&=\mathscr{R}_q(h).
		\end{align*}
		Therefore, combining above equality together with \eqref{definition of R_q} and \eqref{m in terms of R}, we deduce that 
		\begin{align}
			m_{\sqrt{N_j}}\left(\frac{h}{5}\right)=
			\begin{cases}
				c_5\sum\limits_{q:5\mid q}\mathscr{R}_q(h), & \text{for }  h\neq 0,\\~\\
				c_5\sum\limits_{q=1}^{\infty}\mathscr{R}_q(0), & \text{for }  h= 0.
			\end{cases}
		\end{align}
		Note that, for $(5,s)=1$ the equality $\mathscr{R}_{5^ks}(h)=\mathscr{R}_{5^k}(h)\mathscr{R}_{s}(0)$ which is a direct consequence of the multiplicativity of the Ramanujan sum $r_q(h)$ that $r_{qt}(h)=r_q(h)r_t(h)$ while $(q,t)=1$, we have 
		\begin{align}
			m_{\sqrt{N_j}}\left(\frac{h}{5}\right)
			&=c_5\sum\limits_{q:5\mid q}\mathscr{R}_q(h)=c_5\sum\limits_{k\geq 1}\sum\limits_{s:(s,5)=1}\mathscr{R}_{5^ks}(h)=c_5\sum\limits_{k\geq 1}\mathscr{R}_{5^k}(h)\sum\limits_{s:(s,5)=1}\mathscr{R}_{s}(0),\quad\text{for}\;h\neq 0.\label{m as product of R_k and R_s}
		\end{align}
		Also, observe that
		\begin{align}
			\sum\limits_{q=1}^{\infty}\mathscr{R}_q(0)
			&=\sum\limits_{q:5\nmid q}\mathscr{R}_q(0)+\sum\limits_{q:5\mid q}\mathscr{R}_q(0)\nonumber\\
			&=\sum\limits_{s:(s,5)=1}\mathscr{R}_s(0)+\sum\limits_{k\geq 1}\sum\limits_{s:(s,5)=1}\mathscr{R}_{5^ks}(0)\nonumber\\
			&=\sum\limits_{s:(s,5)=1}\mathscr{R}_s(0)\left(1+\sum\limits_{k\geq 1}\mathscr{R}_{5^k}(0)\right)\nonumber\\
			\text{i.e.,}\quad c_5\sum\limits_{s:(s,5)=1}\mathscr{R}_s(0)
			&=m_{\sqrt{N_j}}(0)\left(1+\sum\limits_{k\geq 1}\mathscr{R}_{5^k}(0)\right)^{-1}.\label{m(0) as product of R_k and R_s}
		\end{align}
		Next, we show that for every $k\geq 2$, we have $\mathscr{R}_{5^k}(h)=0$ for all $h\in(\Z/5\Z)^5$. Indeed, for any $k\geq 2$, we note that, every element $x\in(\Z/5^k\Z)^5$ can be uniquely expressed as $x=z+5^{k-1}y$, where $z\in(\Z/5^{k-1}\Z)^5$ and $y\in(\Z/5\Z)^5$. Because $k\geq 2,\;\text{implies}\; 2k-2\geq k$, it follows that $|x|^2\equiv |z|^2+2\cdot 5^{k-1}y\cdot z(\mod 5^k)$ and $h\cdot x\equiv h\cdot z(\mod 5)$. Thus, we have
		\begin{align}
			\mathscr{R}_{5^k}(h)	
			&=\frac{1}{5^{5k}}\sum_{x\in(\Z/5^k\Z)^5}e\left(h\cdot\frac{x}{5}\right)r_{5^k}(|x|^2-N_j)\nonumber\\
			&=\frac{1}{5^{5k}}\sum_{z\in(\Z/5^{k-1}\Z)^5}\sum_{y\in(\Z/5\Z)^5}e\left(h\cdot\frac{z+5^{k-1}y}{5}\right)\sum_{u\in U_{5^k}}e_{5^k}\left(u(|z|^2+2\cdot 5^{k-1}y\cdot z-N_j)\right)\nonumber\\
			&=\frac{1}{5^{5k}}\sum_{u\in U_{5^k}}\sum_{z\in(\Z/5^{k-1}\Z)^5}e\left(h\cdot\frac{z}{5}\right)e_{5^k}\left(u(|z|^2-N_j)\right)\sum_{y\in(\Z/5\Z)^5}e\left(2\cdot u\cdot y\cdot \frac{z}{5}\right)\nonumber\\
			&=\frac{1}{5^{5k}}\sum_{u\in U_{5^k}}\sum_{z\in(\Z/5^{k-1}\Z)^5}e\left(h\cdot\frac{z}{5}\right)e_{5^k}\left(u(|z|^2-N_j)\right)\prod_{j=1}^5\left[\sum_{y_j=0}^{4}e\left(2uy_j\frac{z_j}{5}\right)\right].\label{R_k=0 for k>=2}
		\end{align}
		Since the elementary identity 
		\begin{align}
			\sum_{u=0}^{4}e\left(\frac{mu}{5}\right)
			=
			\begin{cases}
				5, & \text{if } 5\mid m,\\
				0, & \text{otherwise},
			\end{cases}\label{elementary identity}
		\end{align}
		implies that the quantity $\prod_{j=1}^5\left[\sum_{y_j=0}^{4}e\left(2uy_j\frac{z_j}{5}\right)\right]$ is equal to $5^5 1_{\{z\equiv 0(\mod 5)\}}$, the equation \eqref{R_k=0 for k>=2} becomes 
		\begin{align}
			\mathscr{R}_{5^k}(h)
			&=\frac{1}{5^{5k-5}}\sum_{\substack{z\in(\Z/5^{k-1}\Z)^5,\\z\equiv 0(\mod 5)}}e\left(h\cdot\frac{z}{5}\right)\sum_{u\in U_{5^k}}e_{5^k}\left(u(|z|^2-N_j)\right).\label{modified R_{5^k}}
		\end{align}
		Note that, any element $u\in U_{5^k}$ can be uniquely expressed as $u=u_1+5^{k-1}u_2$, where $u_1\in U_{5^{k-1}}$ and $u_2\in \{0,1,2,3,4\}$. Therefore, the inner sum over $u$ in \eqref{modified R_{5^k}} becomes
		\begin{align*}
			\sum_{u_1\in U_{5^{k-1}}}e\left(\frac{u_1(|z|^2-N_j)}{5^k}\right)\sum_{u_2=0}^{4}e\left(\frac{u_2(|z|^2-N_j)}{5}\right)
			=0.
		\end{align*}
		Here we have used the fact that $5$ does not devide $(|z|^2-N_j)$ whenever $N_j$ is not divisible by $5$ and $z\equiv 0(\mod 5)$ along with the identity \eqref{elementary identity}. Hence, we conclude that $\mathscr{R}_{5^k}(h)=0$ for all $k\geq 2$ and $h\in(\Z/5\Z)^5$. Next, an application of \eqref{elementary identity} yields
		\begin{align}
			\mathscr{R}_{5}(h)
			&=\frac{1}{5^{5}}\sum_{x\in(\Z/5\Z)^5}e\left(h\cdot\frac{x}{5}\right)\sum_{u\in U_{5}}e_{5}\left(u(|x|^2-N_j)\right)\nonumber\\
			&=\frac{1}{5^{5}}\sum_{x\in(\Z/5\Z)^5}e\left(h\cdot\frac{x}{5}\right)\left(5\textbf{1}_{\{x:|x|^2\equiv N(\mod 5)\}}-1\right)\nonumber\\
			&=\frac{1}{5^{4}}\sum_{\substack{x\in(\Z/5\Z)^5:\\|x|^2\equiv N_j(\mod 5)}}e\left(h\cdot\frac{x}{5}\right)-\frac{1}{5^{5}}\sum_{x\in(\Z/5\Z)^5}e\left(h\cdot\frac{x}{5}\right)\nonumber\\
			&=\frac{1}{5^{4}}\sum_{\substack{x\in C_a:\\N\equiv a(\mod 5)}}e\left(h\cdot\frac{x}{5}\right)-\textbf{1}_{\{h=0\}}.\label{R_5(0) in terms of C_a}
		\end{align}
		In particular, when $h=0$, we have $\mathscr{R}_{5}(0)=5^{-4}|C_a|-1$. Finally, combining \eqref{m as product of R_k and R_s}, \eqref{m(0) as product of R_k and R_s}, \eqref{R_5(0) in terms of C_a} and the fact that $\mathscr{R}_{5^k}(h)=0$ for all $k\geq 2$ and $h\in(\Z/5\Z)^5$, we obtain the desired estimate as follows:
		\begin{align}
			m_{\sqrt{N_j}}\left(\frac{h}{5}\right)
			&=c_5\sum\limits_{k\geq 1}\mathscr{R}_{5^k}(h)\sum\limits_{s:(s,5)=1}\mathscr{R}_{s}(0)\nonumber\\
			&=\mathscr{R}_{5}(h)\frac{m_{\sqrt{N_j}}(0)}{1+\mathscr{R}_{5}(0)}\nonumber\\
			&=\frac{m_{\sqrt{N_j}}(0)}{|C_a|}\sum_{u\in C_a}e\left(-\frac{h\cdot u}{5}\right),\quad\text{for all }h\neq 0.\label{final equality}
		\end{align}
		Since $C_a=-C_a,$ the sum over $u$ remains unchanged if we replace $u$ by $-u$. Hence, the last equality is justified. Moreover, note that the equality \ref{final equality} is trivial for $h=0$. This completes the proof of \Cref{lem: relation between m(h) and m(0)}.
	\end{proof}

	\subsection*{Acknowledgement}
	\sloppy The authors would like to thank Wenbo Li for suggesting us an ergodic version of the counterexample in \Cref{example1}. We are also grateful to the anonymous referees for their valuable comments and suggestions, which have significantly improved the quality of this paper. Saurabh Shrivastava acknowledges the financial support from Anusandhan National Research Foundation, India under the project ANRF/ARG/2025/000940/MS. Ankit Bhojak is supported by the Science and Engineering Research Board, Department of Science and Technology, Govt. of India, under the scheme National Post-Doctoral Fellowship, file no. PDF/2023/000708. Siddhartha Samanta is supported by IISER Bhopal for PhD fellowship.

\bibliography{bibliography}

@article {BootstrapVariationStein,
    AUTHOR = {Mirek, Mariusz and Stein, Elias M. and Zorin-Kranich, Pavel},
     TITLE = {A bootstrapping approach to jump inequalities and their
              applications},
   JOURNAL = {Anal. PDE},
  FJOURNAL = {Analysis \& PDE},
    VOLUME = {13},
      YEAR = {2020},
    NUMBER = {2},
     PAGES = {527--558},
      ISSN = {2157-5045,1948-206X},
   MRCLASS = {42B25 (42B20 46B06)},
  MRNUMBER = {4078235},
MRREVIEWER = {Lubomira\ G.\ Softova},
       DOI = {10.2140/apde.2020.13.527},
       URL = {https://doi.org/10.2140/apde.2020.13.527},
}

@article {Bootstrap,
    AUTHOR = {S{\l }omian, Wojciech},
     TITLE = {Bootstrap methods in bounding discrete {R}adon operators},
   JOURNAL = {J. Funct. Anal.},
  FJOURNAL = {Journal of Functional Analysis},
    VOLUME = {283},
      YEAR = {2022},
    NUMBER = {9},
     PAGES = {Paper No. 109650, 30},
      ISSN = {0022-1236,1096-0783},
   MRCLASS = {42B25 (42B15)},
  MRNUMBER = {4458228},
       DOI = {10.1016/j.jfa.2022.109650},
       URL = {https://doi.org/10.1016/j.jfa.2022.109650},
}

@article {variationalestimatesradon,
    AUTHOR = {Mirek, Mariusz and Stein, Elias M. and Trojan, Bartosz},
     TITLE = {{$\ell^p(\Bbb Z^d) $}-estimates for discrete operators of
              {R}adon type: variational estimates},
   JOURNAL = {Invent. Math.},
  FJOURNAL = {Inventiones Mathematicae},
    VOLUME = {209},
      YEAR = {2017},
    NUMBER = {3},
     PAGES = {665--748},
      ISSN = {0020-9910,1432-1297},
   MRCLASS = {42B25 (37A25)},
  MRNUMBER = {3681393},
MRREVIEWER = {Jan-Olav\ R\"onning},
       DOI = {10.1007/s00222-017-0718-4},
       URL = {https://doi.org/10.1007/s00222-017-0718-4},
}

@article {Variationalongprimesandpolynomials,
    AUTHOR = {Zorin-Kranich, Pavel},
     TITLE = {Variation estimates for averages along primes and polynomials},
   JOURNAL = {J. Funct. Anal.},
  FJOURNAL = {Journal of Functional Analysis},
    VOLUME = {268},
      YEAR = {2015},
    NUMBER = {1},
     PAGES = {210--238},
      ISSN = {0022-1236,1096-0783},
   MRCLASS = {42A45 (26A45 37A45)},
  MRNUMBER = {3280058},
MRREVIEWER = {Vjekoslav\ Kova\v c},
       DOI = {10.1016/j.jfa.2014.10.018},
       URL = {https://doi.org/10.1016/j.jfa.2014.10.018},
}

@article {DimensionfreeestimatesdiscreteHL,
    AUTHOR = {Bourgain, Jean and Mirek, Mariusz and Stein, Elias M. and
              Wr\'obel, B\l a\.zej},
     TITLE = {Dimension-free estimates for discrete {H}ardy-{L}ittlewood
              averaging operators over the cubes in {$\Bbb Z^d$}},
   JOURNAL = {Amer. J. Math.},
  FJOURNAL = {American Journal of Mathematics},
    VOLUME = {141},
      YEAR = {2019},
    NUMBER = {4},
     PAGES = {857--905},
      ISSN = {0002-9327,1080-6377},
   MRCLASS = {42B25},
  MRNUMBER = {3992568},
MRREVIEWER = {Kenneth\ F.\ Andersen},
       DOI = {10.1353/ajm.2019.0023},
       URL = {https://doi.org/10.1353/ajm.2019.0023},
}

@article {Pointwiseergodictheorems,
    AUTHOR = {Bourgain, Jean},
     TITLE = {Pointwise ergodic theorems for arithmetic sets},
      NOTE = {With an appendix by the author, Harry Furstenberg, Yitzhak
              Katznelson and Donald S. Ornstein},
   JOURNAL = {Inst. Hautes \'Etudes Sci. Publ. Math.},
  FJOURNAL = {Institut des Hautes \'Etudes Scientifiques. Publications
              Math\'ematiques},
    NUMBER = {69},
      YEAR = {1989},
     PAGES = {5--45},
      ISSN = {0073-8301,1618-1913},
   MRCLASS = {28D05 (11B83 11K99 11L03 47A35)},
  MRNUMBER = {1019960},
MRREVIEWER = {F.\ Schweiger},
       URL = {http://www.numdam.org/item?id=PMIHES_1989__69__5_0},
}

@article {variationformartingales,
    AUTHOR = {L\'epingle, D.},
     TITLE = {La variation d'ordre {$p$} des semi-martingales},
   JOURNAL = {Z. Wahrscheinlichkeitstheorie und Verw. Gebiete},
  FJOURNAL = {Zeitschrift f\"ur Wahrscheinlichkeitstheorie und Verwandte
              Gebiete},
    VOLUME = {36},
      YEAR = {1976},
    NUMBER = {4},
     PAGES = {295--316},
   MRCLASS = {60G45},
  MRNUMBER = {420837},
MRREVIEWER = {Norihiko\ Kazamaki},
       DOI = {10.1007/BF00532696},
       URL = {https://doi.org/10.1007/BF00532696},
}

@article{Stein1,
				author={Stein, Elias~M.},
				TITLE = {Maximal functions. {I}. {S}pherical means},
				JOURNAL = {Proc. Nat. Acad. Sci. U.S.A.},
				FJOURNAL = {Proceedings of the National Academy of Sciences of the United
					States of America},
				VOLUME = {73},
				YEAR = {1976},
				NUMBER = {7},
				PAGES = {2174--2175},
				ISSN = {0027-8424},
				MRCLASS = {42A40 (43A85)},
				MRNUMBER = {420116},
				MRREVIEWER = {Alberto\ Torchinsky},
				DOI = {10.1073/pnas.73.7.2174},
				URL = {https://doi.org/10.1073/pnas.73.7.2174},
			}

@article {Krause,
    			AUTHOR = {Cladek, Laura and Krause, Benjamin},
     			TITLE = {Improved endpoint bounds for the lacunary spherical maximal
              operator},
   				JOURNAL = {Anal. PDE},
 			 	FJOURNAL = {Analysis \& PDE},
    			VOLUME = {17},
      			YEAR = {2024},
    			NUMBER = {6},
     			PAGES = {2011--2032},
      			ISSN = {2157-5045,1948-206X},
   				MRCLASS = {42B15 (42B20 42B25 44A12)},
  				MRNUMBER = {4776291},
       			DOI = {10.2140/apde.2024.17.2011},
       			URL = {https://doi.org/10.2140/apde.2024.17.2011},
}

@article {Cook2,
 				AUTHOR = {Cook, Brian},
        		TITLE = {A note on discrete spherical averages over sparse sequences},
   				JOURNAL = {Proc. Amer. Math. Soc.},
  				FJOURNAL = {Proceedings of the American Mathematical Society},
    			VOLUME = {150},
      			YEAR = {2022},
    			NUMBER = {10},
     			PAGES = {4303--4314},
      			ISSN = {0002-9939,1088-6826},
   				MRCLASS = {42B25},
  				MRNUMBER = {4470175},
				MRREVIEWER = {Fayou\ Zhao},
       			DOI = {10.1090/proc/14575},
       			URL = {https://doi.org/10.1090/proc/14575},
}

@article {Calderon,
    			AUTHOR = {Calder\'on, Calixto P.},
     			TITLE = {Lacunary spherical means},
   				JOURNAL = {Illinois J. Math.},
  				FJOURNAL = {Illinois Journal of Mathematics},
    			VOLUME = {23},
      			YEAR = {1979},
    			NUMBER = {3},
     			PAGES = {476--484},
      			ISSN = {0019-2082},
   				MRCLASS = {42B25 (42B20)},
  				MRNUMBER = {537803},
				MRREVIEWER = {Douglas\ Kurtz},
       			URL = {http://projecteuclid.org/euclid.ijm/1256048109},
}

@article{MSW2002,
				AUTHOR = {Magyar, A. and Stein, E. M. and Wainger, S.},
				TITLE = {Discrete analogues in harmonic analysis: spherical averages},
				JOURNAL = {Ann. of Math. (2)},
				FJOURNAL = {Annals of Mathematics. Second Series},
				VOLUME = {155},
				YEAR = {2002},
				NUMBER = {1},
				PAGES = {189--208},
				ISSN = {0003-486X,1939-8980},
				MRCLASS = {42B25 (11K70 43A07)},
				MRNUMBER = {1888798},
				MRREVIEWER = {Loukas\ Grafakos},
				DOI = {10.2307/3062154},
				URL = {https://doi.org/10.2307/3062154},
			}

@article{Lacey1,
				AUTHOR = {Kesler, Robert and Lacey, Michael T. and Mena, Dar\'{\i}o},
				TITLE = {Sparse bounds for the discrete spherical maximal functions},
				JOURNAL = {Pure Appl. Anal.},
				FJOURNAL = {Pure and Applied Analysis},
				VOLUME = {2},
				YEAR = {2020},
				NUMBER = {1},
				PAGES = {75--92},
				ISSN = {2578-5885,2578-5893},
				MRCLASS = {11K70 (42B25)},
				MRNUMBER = {4041278},
				MRREVIEWER = {Stefan\ Steinerberger},
				DOI = {10.2140/paa.2020.2.75},
				URL = {https://doi.org/10.2140/paa.2020.2.75},
			}

@article{LaceyLacunarySpherical,
				AUTHOR = {Kesler, Robert and Lacey, Michael T. and Mena Arias,
					Dar\'{\i}o},
				TITLE = {Lacunary discrete spherical maximal functions},
				JOURNAL = {New York J. Math.},
				FJOURNAL = {New York Journal of Mathematics},
				VOLUME = {25},
				YEAR = {2019},
				PAGES = {541--557},
				ISSN = {1076-9803},
				MRCLASS = {42B25 (11L05 11P55 42B20)},
				MRNUMBER = {3982252},
				MRREVIEWER = {Ioannis\ R.\ Parissis},
			}

@article{Bourgain2,
				AUTHOR = {Bourgain, Jean},
				TITLE = {Estimations de certaines fonctions maximales},
				JOURNAL = {C. R. Acad. Sci. Paris S\'{e}r. I Math.},
				FJOURNAL = {Comptes Rendus des S\'{e}ances de l'Acad\'{e}mie des Sciences.
					S\'{e}rie I. Math\'{e}matique},
				VOLUME = {301},
				YEAR = {1985},
				NUMBER = {10},
				PAGES = {499--502},
				ISSN = {0249-6291},
				MRCLASS = {42B25},
				MRNUMBER = {812567},
			}

@article{Cook&Hughes1,
				AUTHOR = {Brian Cook and Kevin Hughes},
				TITLE = {Bounds for lacunary maximal functions given by
					{B}irch-{M}agyar averages},
				JOURNAL = {Trans. Amer. Math. Soc.},
				FJOURNAL = {Transactions of the American Mathematical Society},
				VOLUME = {374},
				YEAR = {2021},
				NUMBER = {6},
				PAGES = {3859--3879},
				ISSN = {0002-9947,1088-6850},
				MRCLASS = {11D72 (42B25)},
				MRNUMBER = {4251215},
				MRREVIEWER = {Zane\ Kun\ Li},
				DOI = {10.1090/tran/8152},
				URL = {https://doi.org/10.1090/tran/8152},
			}

@article{Magyar2,
				AUTHOR = {Magyar, Akos},
				TITLE = {Diophantine equations and ergodic theorems},
				JOURNAL = {Amer. J. Math.},
				FJOURNAL = {American Journal of Mathematics},
				VOLUME = {124},
				YEAR = {2002},
				NUMBER = {5},
				PAGES = {921--953},
				ISSN = {0002-9327,1080-6377},
				MRCLASS = {37A45 (11D45 11K55 28D15 37A30)},
				MRNUMBER = {1925339},
				MRREVIEWER = {Thomas\ Ward},
				URL =
				{http://muse.jhu.edu/journals/american_journal_of_mathematics/v124/124.5magyar.pdf},
			}

@article{DuoanRubio,
	AUTHOR = {Duoandikoetxea, Javier and Rubio de Francia, Jos\'e L.},
	TITLE = {Maximal and singular integral operators via {F}ourier
	transform estimates},
	JOURNAL = {Invent. Math.},
	FJOURNAL = {Inventiones Mathematicae},
	VOLUME = {84},
	YEAR = {1986},
	NUMBER = {3},
	PAGES = {541--561},
	ISSN = {0020-9910,1432-1297},
	MRCLASS = {42B25 (42B20)},
	MRNUMBER = {837527},
	MRREVIEWER = {Douglas\ Kurtz},
	DOI = {10.1007/BF01388746},
	URL = {https://doi.org/10.1007/BF01388746},
}

@article {Ionescu,
    AUTHOR = {Ionescu, Alexandru D.},
     TITLE = {An endpoint estimate for the discrete spherical maximal
              function},
   JOURNAL = {Proc. Amer. Math. Soc.},
  FJOURNAL = {Proceedings of the American Mathematical Society},
    VOLUME = {132},
      YEAR = {2004},
    NUMBER = {5},
     PAGES = {1411--1417},
      ISSN = {0002-9939,1088-6826},
   MRCLASS = {42B25 (39A12)},
  MRNUMBER = {2053347},
MRREVIEWER = {Michael\ A.\ Perelmuter},
       DOI = {10.1090/S0002-9939-03-07207-1},
       URL = {https://doi.org/10.1090/S0002-9939-03-07207-1},
}

@article {variationalinequalitycontinuous,
    AUTHOR = {Jones, Roger L. and Seeger, Andreas and Wright, James},
     TITLE = {Strong variational and jump inequalities in harmonic analysis},
   JOURNAL = {Trans. Amer. Math. Soc.},
  FJOURNAL = {Transactions of the American Mathematical Society},
    VOLUME = {360},
      YEAR = {2008},
    NUMBER = {12},
     PAGES = {6711--6742},
      ISSN = {0002-9947,1088-6850},
   MRCLASS = {42B20 (42B15 42B25)},
  MRNUMBER = {2434308},
MRREVIEWER = {Javier\ Duoandikoetxea},
       DOI = {10.1090/S0002-9947-08-04538-8},
       URL = {https://doi.org/10.1090/S0002-9947-08-04538-8},
}

@article{bcss,
 author = {Bhojak, Ankit and Choudhary, Surjeet Singh and Samanta, Siddhartha and                    Shrivastava, Saurabh},
  title = {Sparse bounds for discrete maximal functions associated with Birch–Magyar                  averages},
journal = {Mathematika},
 volume = {71},
 number = {4},
  pages = {e70048},
    doi = {https://doi.org/10.1112/mtk.70048},
    url = {https://londmathsoc.onlinelibrary.wiley.com/doi/abs/10.1112/mtk.70048},
   year = {2025}
}

@article {endPoint,
    AUTHOR = {Bhojak, Ankit and Shrivastava, Saurabh},
     TITLE = {Endpoint Variation and jump inequalities for rough singular integrals},
      YEAR = {2026},
       DOI = {https://doi.org/10.48550/arXiv.2602.21888},
       URL = {https://doi.org/10.48550/arXiv.2602.21888},
}

@article {endPointMaximal,
    AUTHOR = {Lee, Sanghyuk and Li, Ji and Liang, Chong-Wei and Shen, Chun-Yen},
     TITLE = {On the endpoint estimate for discrete spherical average over sparse sequences},
      YEAR = {2026},
       DOI = {https://doi.org/10.48550/arXiv.2608.03004},
       URL = {https://doi.org/10.48550/arXiv.2608.03004},
}

@book {AnalyticNumberTheoryBOOK,
    AUTHOR = {Iwaniec, Henryk and Kowalski, Emmanuel},
     TITLE = {Analytic number theory},
    SERIES = {American Mathematical Society Colloquium Publications},
    VOLUME = {53},
 PUBLISHER = {American Mathematical Society, Providence, RI},
      YEAR = {2004},
     PAGES = {xii+615},
      ISBN = {0-8218-3633-1},
   MRCLASS = {11-02 (11Fxx 11Lxx 11Mxx 11Nxx)},
  MRNUMBER = {2061214},
MRREVIEWER = {K.\ Soundararajan},
       DOI = {10.1090/coll/053},
       URL = {https://doi.org/10.1090/coll/053},
}

@article {VariationBoundSpherical,
    AUTHOR = {Beltran, David and Oberlin, Richard and Roncal, Luz and
              Seeger, Andreas and Stovall, Betsy},
     TITLE = {Variation bounds for spherical averages},
   JOURNAL = {Math. Ann.},
  FJOURNAL = {Mathematische Annalen},
    VOLUME = {382},
      YEAR = {2022},
    NUMBER = {1-2},
     PAGES = {459--512},
      ISSN = {0025-5831,1432-1807},
   MRCLASS = {42B15 (42B25)},
  MRNUMBER = {4377310},
MRREVIEWER = {Sundaram\ Thangavelu},
       DOI = {10.1007/s00208-021-02218-2},
       URL = {https://doi.org/10.1007/s00208-021-02218-2},
}

@article {JW2004,
    AUTHOR = {Jones, Roger L. and Wang, Gang},
     TITLE = {Variation inequalities for the {F}ej\'er and {P}oisson
              kernels},
   JOURNAL = {Trans. Amer. Math. Soc.},
  FJOURNAL = {Transactions of the American Mathematical Society},
    VOLUME = {356},
      YEAR = {2004},
    NUMBER = {11},
     PAGES = {4493--4518},
      ISSN = {0002-9947,1088-6850},
   MRCLASS = {42A24 (26D05 42A55 42A85)},
  MRNUMBER = {2067131},
MRREVIEWER = {Alexander\ V.\ Tovstolis},
       DOI = {10.1090/S0002-9947-04-03397-5},
       URL = {https://doi.org/10.1090/S0002-9947-04-03397-5},
}

@article {Qian1998,
    AUTHOR = {Qian, Jinghua},
     TITLE = {The {$p$}-variation of partial sum processes and the empirical
              process},
   JOURNAL = {Ann. Probab.},
  FJOURNAL = {The Annals of Probability},
    VOLUME = {26},
      YEAR = {1998},
    NUMBER = {3},
     PAGES = {1370--1383},
      ISSN = {0091-1798,2168-894X},
   MRCLASS = {60F05 (60F15 60J15 62G30)},
  MRNUMBER = {1640349},
MRREVIEWER = {Rimas\ Norvai\v sa},
       DOI = {10.1214/aop/1022855756},
       URL = {https://doi.org/10.1214/aop/1022855756},
}

@article {MSW2022,
    AUTHOR = {Mirek, Mariusz and Szarek, Tomasz Z. and Wright, James},
     TITLE = {Oscillation inequalities in ergodic theory and analysis:
              one-parameter and multi-parameter perspectives},
   JOURNAL = {Rev. Mat. Iberoam.},
  FJOURNAL = {Revista Matem\'atica Iberoamericana},
    VOLUME = {38},
      YEAR = {2022},
    NUMBER = {7},
     PAGES = {2249--2284},
      ISSN = {0213-2230,2235-0616},
   MRCLASS = {37A30 (37A46 42B25)},
  MRNUMBER = {4525137},
       DOI = {10.4171/rmi/1383},
       URL = {https://doi.org/10.4171/rmi/1383},
}

@article {MSZ2020,
    AUTHOR = {Mirek, Mariusz and Stein, Elias M. and Zorin-Kranich, Pavel},
     TITLE = {Jump inequalities via real interpolation},
   JOURNAL = {Math. Ann.},
  FJOURNAL = {Mathematische Annalen},
    VOLUME = {376},
      YEAR = {2020},
    NUMBER = {1-2},
     PAGES = {797--819},
      ISSN = {0025-5831,1432-1807},
   MRCLASS = {60E15 (46B70 46M35 60G44)},
  MRNUMBER = {4055178},
MRREVIEWER = {Stefan\ Geiss},
       DOI = {10.1007/s00208-019-01889-2},
       URL = {https://doi.org/10.1007/s00208-019-01889-2},
}
\end{document}